\documentclass[11pt, reqno]{amsart}

\usepackage{amssymb,mathrsfs,graphicx,extpfeil}
\usepackage{amsmath,amsfonts,amssymb,amscd,amsthm,bbm}
\usepackage{dsfont}
\usepackage{epsfig}
\usepackage{comment}
\usepackage{caption}
\usepackage{subcaption}
\usepackage{graphicx,colortbl}
\usepackage{bm}
\usepackage{tikz}
\usepackage{tikz-cd}
\usetikzlibrary{matrix,shapes,arrows,positioning, chains}
\usepackage[unicode,hypertexnames=false]{hyperref}
\hypersetup{colorlinks,breaklinks,linkcolor=blue,urlcolor=blue,anchorcolor=blue,citecolor=blue}
\usepackage[nameinlink, capitalise, noabbrev]{cleveref}
\usepackage[shortlabels]{enumitem}
\usepackage{kotex}
\usepackage[autostyle]{csquotes}
   \MakeOuterQuote{"}

\title[]{Emergent behaviors of Winfree oscillators on special orthogonal group}

\author[Ha]{Seung-Yeal Ha}
\address[Seung-Yeal Ha]{\newline Department of Mathematical Sciences and Research Institute of Mathematics \newline Seoul National University, Seoul 08826, Republic of Korea}
\email{syha@snu.ac.kr}

\author[Lee]{Chaejoo Lee}
\address[Chaejoo Lee]{\newline Department of Mathematical Sciences\newline Seoul National University, Seoul 08826, Republic of Korea}
\email{leechaeju@snu.ac.kr}

\author[Lee]{Eunjun Lee}
\address[Eunjun Lee]{\newline Department of Mathematical Sciences\newline Seoul National University, Seoul 08826, Republic of Korea}
\email{eunjun.lee@snu.ac.kr}

\author[Lee]{Jaemoon Lee}
\address[Jaemoon Lee]{\newline Department of Mathematical Sciences\newline Seoul National University, Seoul 08826, Republic of Korea}
\email{dlwoans0001@snu.ac.kr}

\author[Ryoo]{Seung-Yeon Ryoo}
\address[Seung-Yeon Ryoo]{Division of Physics, Mathematics and Astronomy,
\newline California Institute of Technology,
Pasadena, CA 91125, USA}
\email{sryoo@caltech.edu}

\newtheorem{theorem}{Theorem}[section]
\newtheorem{lemma}{Lemma}[section]
\newtheorem{corollary}{Corollary}[section]
\newtheorem{proposition}{Proposition}[section]
\newtheorem{remark}{Remark}[section]
\newtheorem{definition}[theorem]{Definition}

\newcommand{\bbr}{\mathbb R}

\newcommand{\bbz}{\mathbb Z}

\newcommand{\tr}{\mathrm{tr}\, }

\DeclareMathOperator*{\argmin}{argmin}

\begin{document}
\tikzstyle{block} = [rectangle, draw, 
text width=15em, text centered, rounded corners, minimum height=3em]
\tikzstyle{line} = [draw, -latex']

\date{\today}

\subjclass{34C40, 34C15, 34D06} 
\keywords{Winfree model, special orthogonal group, synchronization, emergent dynamics, oscillator death}
\thanks{\textbf{Acknowledgment.}
The work of S.-Y. Ha was supported by the NRF grant (RS\_2025-00514472).}
\begin{abstract}
We propose a generalized matrix-valued synchronization model which can be regarded as matrix generalization of the classical Winfree model to the special orthogonal group, and we provide several sufficient frameworks leading to the emergent behaviors of the Winfree matrix model. For $SO(2)$ case, the proposed model reduces to the classical Winfree model. For the general (non-identical) case, we prove the existence of a positively invariant trapping region, establish a leader--follower mechanism in which sufficiently strong coupling draws all oscillators into a neighborhood of the identity whenever at least one oscillator is initially nearby, and show $\ell^1$-exponential stability of solutions, from which we deduce existence, uniqueness, and exponential convergence to an equilibrium. In the identical-oscillator regime, we show that complete state synchronization and oscillator death both occur exponentially fast with an explicit decay rate, and we classify all equilibrium configurations as solutions to a fixed-point equation for the mean influence.
\end{abstract}

\maketitle

\centerline{\date}


\section{Introduction} \label{sec:1}
\setcounter{equation}{0}
Synchronization arises across a wide range of physical and biological systems, from collective animal motion to rhythmic activity in engineering. Well-known examples include flocking in birds and fish \cite{flierl1999individuals,vicsek1995novel}, synchronous flashing of fireflies \cite{buck1966biology}, collective dynamics of cardiac pacemaker cells \cite{peskin1975mathematical}, and frequency synchronization in power networks \cite{chiang2011direct,kundur2007power,sauer2017power}. A variety of mathematical models have been proposed to capture such emergent behavior. To name a few, the Winfree and Kuramoto models for coupled oscillators \cite{winfree1967biological,kuramoto1975international}, and the Vicsek and Cucker--Smale models for flocking \cite{cucker2007emergent,vicsek1995novel}. Among them, we are interested in the high-dimensional counterpart of the Winfree model  \cite{winfree1967biological}. To set up the stage, we begin with the brief description of the Winfree model. 

Let $\theta_i(t) \in \mathbb{S}^1$ be the phase of the $i$-th Winfree oscillator; lifting to $\mathbb{R}$ modulo $2\pi$. Then the dynamics of phase is governed by the following Cauchy problem for the coupled system of ordinary differential equations:
\begin{equation} \label{A-1}
\begin{cases}
\displaystyle \dot{\theta}_i = \nu_i + \kappa \Big( \frac{1}{N} \sum_{j=1}^{N} I(\theta_j) \Big) S(\theta_i), \quad t > 0,  \\
\displaystyle  \theta_i \Big|_{t = 0} = \theta_i^0,  \quad i \in [N] := \{1, \cdots, N \}, 
\end{cases}
\end{equation}
where $\nu_i$ is the stationary natural frequency of the $i$-th oscillator, $\kappa > 0$ is the coupling strength, and $S,I: \mathbb{R} \to \mathbb{R}$ are $2\pi$-periodic sensitivity and influence functions. Depending on the choice of $S, I$, the natural frequencies, and $\kappa$, this model exhibits a rich variety of collective behaviors, including incoherence \cite{ariaratnam2001phase}, partial entrainment \cite{ha2017emergence}, phase-locked states \cite{oukil2017synchronization,oukil2019invariant}, and oscillator-death equilibria \cite{ha2015emergence,ha2016emergent,ryoo2026oscillator} etc.

Among others, we are interested in the high-dimensional generalization of the classical Winfree model \eqref{A-1} on matrix Lie group and  the {\it "relaxation dynamics to equilibrium"} in which all oscillators converge to a common distinguished phase. This viewpoint naturally motivates extending the Winfree model to higher-dimensional spaces: one fixes a preferred attraction point $p$
 on a manifold $M$, and oscillator death becomes convergence of all oscillators toward a neighborhood of $p$. Recently, Hansol Park \cite{park2021generalization} carried out this program on the sphere $\mathbb{S}^{n-1}$ with $n \geq 3$. Let $x_i = x_i(t) \in {\mathbb R}^{n}$ be the state of the $i$-th Winfree oscillator on the unit sphere ${\mathbb S}^{n-1}$. Then, its temporal dynamics is governed by the Winfree sphere model:
\begin{equation}\notag\label{A-1-1}
\begin{cases}
\displaystyle \dot{x}_i = \Omega_i x_i + \kappa \eta(x_i) \Big( \frac{1}{N} \sum_{j=1}^{N} I(x_j) \Big) (e - \langle x_i, e \rangle x_i), \quad t > 0, \\
\displaystyle x_i \Big|_{t = 0} = x_i^0, \quad i \in [N], 
\end{cases}
\end{equation}
where $\Omega_i$ is a $n \times n$ anti-symmetric matrix, $e \in \mathbb{S}^{n-1} (\subset {\mathbb R}^{n})$ is the attraction point, and $\eta(x_i) = S(x_i) / \|e - \langle x_i, e \rangle x_i\|$. The forcing term is projected onto the tangent space $T_{x_i}\mathbb{S}^{n-1}$ so that trajectories remain on the unit sphere.

The purpose of this paper is to extend the Winfree model to the special orthogonal group $SO(n)$, and then provide several frameworks leading to the emergent collective behaviors. This setting is natural for several reasons. First, many physical systems of interest — including attitude synchronization of rigid bodies, coupled quantum oscillators, and multi-dimensional biological rhythms — are described by phases taking values in higher-dimensional spaces rather than the circle $\mathbb{S}^1$. 

Second, by the Peter--Weyl theorem, every compact Lie group\footnote{One may ask what happens for noncompact Lie groups. It turns out that the analysis done here is essentially local and may be extended to any dimensional Riemannian manifold; see our companion paper \cite{ha2026convergence}.} embeds into a unitary group, which in turn embeds into an orthogonal group of higher dimension; thus by the locality of our analysis, $SO(n)$ provides a general and natural framework for Winfree-type dynamics on compact Lie groups. We also note that synchronization on Lie groups has attracted considerable recent interest through the non-abelian Kuramoto model of Lohe \cite{lohe2009non}, whose stability properties have been extensively studied in \cite{deville2019synchronization,ha2016emergencelohe,ha2017emergent,ryoo2025asymptotic}; however, Winfree-type interactions in this setting remain largely unexplored, the only works known to the authors being \cite{ha2021collective,park2021generalization}. We refer to \cite{ha2016collective} for a survey on the related topic of quantum synchronization,  \cite{fetecau2021intrinsic} for related work on aggregation models on $SO(3)$ and \cite{lynch2017modern} for the applications to robotics. Next, we return to the design problem for Winfree oscillators on $SO(n)$.

Let $R_i \in SO(n)$ denote the state (or generalized phase) of the $i$-th oscillator, and let $I_n$ denote the $n \times n$ identity matrix , which serves as the attraction point (this does not hurt the generality, see Lemma \ref{L2.2}). The Cauchy problem for the Winfree matrix model on $SO(n)$ reads as follows:
\begin{equation} \label{A-2}
\begin{cases}
\displaystyle \dot{R}_i = \Omega_i R_i + \frac{\kappa}{2} \Big( \frac{1}{N} \sum_{j=1}^{N} I(R_j)\Big)(I_n - R_i^2), \quad t > 0, \\
\displaystyle R_i \Big|_{t = 0} = R_i^0 \in SO(n), \quad \Omega_i \in \mathfrak{so}(n), \quad i \in [N].
\end{cases}
\end{equation}
The term $(I_n - R_i^2)$ plays the role of the projection onto the tangent space of $SO(n)$ at $R_i$, pulling each oscillator toward the identity, while the influence function $I(R_j)$ encodes the strength of that pull as a function of the $j$-th oscillator's distance from $I_n$. Our main contributions are three-fold:
\smallskip

\noindent\textbf{(1) Model formulation on matrix Lie groups:} We propose a Winfree-type interacting oscillator model on general matrix Lie groups and derive an explicit formulation for $SO(n)$. We also verify that for $n= 2$, the model reduces to the classical Winfree model on $\mathbb{S}^1$ (see Section \ref{sec:2.2} and Section \ref{sec:2.3}).

\smallskip
\noindent\textbf{(2) Emergent dynamics and stability:} We analyze the collective behavior of the Winfree model on $SO(n)$. Under suitable conditions on the coupling strength and initial data, we prove the existence of a trapping region. We establish a leader--follower (herding) mechanism: if at least one oscillator lies initially close to the attraction point $I_n$, then sufficiently strong coupling draws all oscillators into a small neighborhood of $I_n$ (see Proposition \ref{P3.2} and Corollary \ref{C3.1}). Finally, we establish $\ell^1$-exponential stability of solutions (see Theorem \ref{T3.1}), from which we deduce existence and uniqueness of an equilibrium near $I_n$ and exponential convergence to it (see Section \ref{sec:4} and Theorem \ref{T4.1}).

\smallskip
\noindent\textbf{(3) Identical-oscillator regime:} When all oscillators share the same natural frequency matrix $\Omega_i = \Omega$ for all $i \in [N]$, we prove that relaxation to equilibrium occurs exponentially fast (see Theorem \ref{T5.1}). We also classify all equilibrium configurations in this regime (see Section \ref{sec:5.2}).

\vspace{0.5cm}

The rest of the paper is organized as follows. In Section~\ref{sec:2}, we recall relevant concepts and results for the classical Winfree and sphere models, and collects basic facts on Lie groups and Lie algebras, and present abstract Winfree model on a Lie group $G$. In Section~\ref{sec:3}, we provide an existence of a trapping region and uniform stability of the Winfree matrix model $SO(n)$. In Section~\ref{sec:4}, we establish the existence of equilibria. In Section~\ref{sec:5}, we provide the emergence of complete state synchronization for a homogeneous ensemble. Finally, Section \ref{sec:6} is devoted to a brief summary of our main results and some remaining issues to be explored for a future work. In Appendix \ref{App-A}, we provide proofs for several assertions in Lemma \ref{L2.3} and \ref{L4.1}.

\vspace{0.5cm}

\noindent {\bf Gallery of Notations}:~We denote the standard Euclidean norm of $x \in \mathbb{R}^n$ by $\|x\|_{\ell^2}$. For a Banach  
space $X$ with norm $\|\cdot\|_X$, we define the $\ell^1(X)$ norm of 
$\Phi = (\phi_1, \dots, \phi_N) \in X^N$ by
\[
\|\Phi\|_{\ell^1(X)} = \sum_{i=1}^{N} \|\phi_i\|_X.
\]
For matrices $A, B \in \mathbb{R}^{n \times n}$, we use the constant multiple of Frobenius
inner product and its associated Frobenius norm:
\begin{equation}\label{A-3}
\langle A, B \rangle := \frac12\operatorname{tr}(A^T B), 
\qquad \|A\| := \sqrt{\langle A, A \rangle}.
\end{equation}
This modified Frobenius inner product(Hilbert-Schmidt inner product) was introduced in \cite{fetecau2021intrinsic} as a Riemannian metric on the tangent space of $SO(3)$. We use it as the inner product on $\bbr^{n\times n}$. This choice simplifies the distance formula appearing in Lemma \ref{L2.3}.

\section{Preliminaries}\label{sec:2}
\setcounter{equation}{0}
In this section, we recall relevant definitions for the descriptions of the classical Winfree model and the Winfree sphere model on ${\mathbb S}^1$ and ${\mathbb S}^{n-1}$ with $n \geq 3$, respectively, and present basic facts on Lie groups and Lie algebras to be used throughout the rest of this paper. 

\subsection{Elementary jargons} \label{sec:2.1} In this subsection, we briefly review several elementary concepts to be used in later sections. The Winfree model \eqref{A-1} exhibits a variety of collective 
behaviors depending on the choice of sensitivity and influence functions $S$ and $I$ 
and the coupling strength $\kappa$. First, we recall the several concepts of collective behaviors of Winfree oscillators as follows. 
\begin{definition} \label{D2.1}
Let $\Theta = \{ \theta_i \}_{i=1}^{N}$ and $\Theta^\infty = \{ \theta_i^\infty \}_{i=1}^{N}$ be a global solution and an equilibrium to the Winfree model  \eqref{A-1}, respectively:
\begin{enumerate}
   \item
   $\Theta^{\infty}$ is an equilibrium if the following relations hold.
   \[  \nu_i + \kappa \Big( \frac{1}{N} \sum_{j=1}^{N} I(\theta^{\infty}_j) \Big) S(\theta^{\infty}_i) = 0, \quad \forall~i \in [N]. \]
   \item $\Theta$ relaxes to  the equilibrium $\Theta^{\infty}$ if the following relations hold. 
    \[
    \lim_{t \to \infty} |\theta_i(t) - \theta_i^\infty| = 0, ~~\forall\, i \in [N].
    \]
    \item $\Theta$ exhibits complete phase synchronization if  the following relation holds.
    \[
    \lim_{t \to \infty} \max_{i, j \in [N]} |\theta_i(t) - \theta_j(t)| = 0.
    \]
    \item $\Theta$ exhibits complete (frequency) synchronization if  the following relation holds.
    \[
    \lim_{t \to \infty} \max_{i, j \in [N]} |{\dot \theta}_i(t) - {\dot \theta}_j(t)| = 0.
    \]
 \end{enumerate}
\end{definition}
\begin{remark}
The above three modes of collective behaviors — convergence to a common equilibrium, zero convergence of relative phases
and relative frequencies — are primary phenomena appearing in the oscillator ensemble  that we seek to 
generalize to the special orthogonal group.  We note that convergence to equilibrium corresponds to the 
oscillator death and it has been studied in detail in \cite{ha2015emergence, ha2016emergent, ryoo2026oscillator}.
\end{remark}
Next, we recall the relevant notions from the theory of Lie groups and Lie 
algebras.
\begin{definition}\label{D2.2}
Let $G$ be a group with the identity element $e$. 
\begin{enumerate}
\item
A Lie group is a group $G$ equipped with the structure of a smooth 
manifold such that the multiplication map $G \times G \to G$, 
$(g, h) \mapsto gh$, and the inversion map $G \to G$, $g \mapsto g^{-1}$, 
are both smooth. 
\vspace{0.1cm}
\item
The associated Lie algebra $\mathfrak{g} := T_e G$ is the 
tangent space of $G$ at the identity element $e \in G$, endowed with the Lie bracket 
$[\cdot, \cdot] : \mathfrak{g} \times \mathfrak{g} \to \mathfrak{g}$ 
satisfying anti-symmetry and the Jacobi identity:
\[ [X, Y] = -[Y, X], \quad 
[X, [Y, Z]] + [Y, [Z, X]] + [Z, [X, Y]] = 0, 
\quad \forall\, X, Y, Z \in \mathfrak{g}.
\]
\end{enumerate}
\end{definition}
\begin{remark}
Below, we provide several comments.
\begin{enumerate}
\item
For the matrix Lie group $G \subset GL(n, \mathbb{R})$, the Lie bracket reduces 
to the matrix commutator $[X, Y] = XY - YX$.
\item
The principal example for the pair of Lie group and its associated Lie algebra is the special orthogonal group and its Lie algebra:
\[
SO(n) = \{Q \in \mathbb{R}^{n \times n} : Q^T Q = I_n,\ \det Q = 1\},
\quad
\mathfrak{so}(n) = \{X \in \mathbb{R}^{n \times n} : X^T = -X\}.
\]
\end{enumerate}
\end{remark}

\subsection{Abstract Winfree model on a matrix Lie group $G$} \label{sec:2.2} 
Recall that the Lie algebra $\mathfrak{g}$ of a matrix Lie group $G$ is a 
subspace of $\mathbb{R}^{n \times n}$, hence it is closed, and so it admits a unique 
orthogonal projection with respect to the inner product defined in the gallery of 
notations \eqref{A-3} at the end of Section \ref{sec:1}. Namely, the projection map ${\mathbb P}$ onto a closed subspace ${\mathcal Y}$ of 
$\mathbb{R}^{n \times n}$ is defined as follows. For $X  \in \mathbb{R}^{n \times n}$, 
\begin{equation} \label{B-1}
    {\mathbb P}_{\mathcal Y} X = \argmin_{Y \in {\mathcal Y}}\ \|X- Y\|^2 = \argmin_{Y \in {\mathcal Y}}\ \langle X-Y, X-Y \rangle.
\end{equation}
As a concrete example, we consider the orthogonal projection from $\mathbb{R}^{n \times n}$ on $\mathfrak{so}(n)$. More precisely, for any $X \in \mathbb{R}^{n \times n}$,  the orthogonal projection is explicitly given as follows.
\begin{equation*} \label{B-2}
{\mathbb P}_{\mathfrak{so}(n)}(X) = \frac{X - X^T}{2}.
\end{equation*}

Next, we present the abstract Winfree model on the matrix Lie group. Note that the Winfree model \eqref{A-1} on $\mathbb{S}^1$ consists of two 
contributions to the phase velocity of each oscillator: a natural frequency 
term $\nu_i$, and an attraction term that pulls $\theta_i$ toward a 
distinguished phase. To generalize this to a matrix Lie group $G$, we 
must ensure that the velocity $\dot{R}_i$ remains in the tangent space 
$T_{R_i}G$ at all times, so that $R_i(t)$ stays on the manifold $G$. Since $G$ is a matrix Lie group, the tangent space at $R_i$ can be identified with right translation of the Lie algebra $\mathfrak{g}$:
\[
T_{R_i}G = \mathfrak{g}R_i.
\]
The natural frequency term is therefore $\Omega_i R_i$ for a fixed 
$\Omega_i \in \mathfrak{g}$. For the attraction term, we fix a preferred 
point $Q \in G$ and project the displacement $Q - R_i$ onto $T_{R_i}G$:
\[
{\mathbb P}_{\mathfrak{g}R_i}(Q - R_i).
\]
Following the philosophy of the Winfree model, the strength of this 
attraction is modulated by a sensitivity function $\eta(R_i)$ and 
weighted by the influence $I(R_j)$ of each other oscillator $j \in [N]$. This 
leads to the abstract Winfree model on $G$:
\begin{equation}\label{B-3}
\begin{cases}
\displaystyle \dot{R}_i = \Omega_i R_i 
+ \kappa \eta(R_i)  \Big( \frac{1}{N} \sum_{j=1}^{N} I(R_j)\Big)
{\mathbb P}_{\mathfrak{g}R_i}(Q - R_i), \quad t > 0, \\[6pt]
R_i \Big|_{t = 0} = R_i^0 \in G, \quad \Omega_i \in \mathfrak{g}, \quad ~i \in [N].
\end{cases}
\end{equation}
Since $\dot{R}_i \in T_{R_i}G$ by construction and $R_i^0 \in G$, we have
\begin{equation*} 
R_i(t) \in G \quad \text{for all}\ t \geq 0. 
\end{equation*}
For a general matrix Lie group, the projection ${\mathbb P}_{\mathfrak{g}R_i}$ might not have a simple form in general. However, for $G = SO(n)$, we can compute it explicitly as follows.
\begin{lemma} \label{L2.1}
For a fixed $R \in SO(n)$, the orthogonal projection onto the tangent 
space $T_R SO(n) = \mathfrak{so}(n) R$ satisfies the right-invariance property:
\[
{\mathbb P}_{\mathfrak{so}(n) R}(X) = {\mathbb P}_{\mathfrak{so}(n)}(XR^{T})\, R 
= \frac{XR^T - RX^T}{2}\, R.
\]
\end{lemma}
\begin{proof}
For given $R\in SO(n)$ and $X\in \bbr^{n\times n}$, we use \eqref{B-1} to find 
\begin{equation*}
    \begin{aligned}
        {\mathbb P}_{\mathfrak{so}(n)R}(X) &= \argmin_{S\in \mathfrak{so}(n)R} \langle X-S, X-S  \rangle \\
        &=\argmin_{S\in\mathfrak{so}(n) R} \tr\Big((X-S)^T(X-S)\Big) \\
        &=\Big[ \argmin_{\tilde S\in \mathfrak{so}(n)} \tr \Big(R^T(XR^T-\tilde S)^T(XR^T-\tilde S)R\Big)\Big] R \\
        &= \Big[\argmin_{\tilde S\in \mathfrak{so}(n)} \tr \Big((XR^T-\tilde S)^T (XR^T-\tilde S)\Big)\Big] R \\
        &={\mathbb P}_{\mathfrak{so}(n)}(XR^T)R.
    \end{aligned}
\end{equation*}
\end{proof}
\noindent We apply Lemma \ref{L2.1} to the attraction term with $X = Q - R_i$  to find 
\begin{align}
\begin{aligned} \label{B-4}
{\mathbb P}_{\mathfrak{so}(n)R_i}(Q - R_i) &= {\mathbb P}_{\mathfrak{so}(n)}\big((Q-R_i)R_i^T\big)\,R_i  \\
 &= \frac{(Q - R_i)R_i^T - R_i(Q - R_i)^T}{2}\,R_i = \frac{QR_i^T - R_iQ^T}{2}\,R_i.
\end{aligned}
\end{align}
Substituting \eqref{B-4} into \eqref{B-3}, we obtain the Winfree matrix model on $SO(n)$:
\begin{equation}\label{B-5}
\begin{cases}
\displaystyle \dot{R}_i R_i^T = \Omega_i 
+ \frac{\kappa}{2} \eta(R_i)  \Big(\frac{1}{N} \sum_{j=1}^{N} I(R_j)\Big)
(QR_i^T - R_iQ^T), \quad t > 0,  \\[6pt]
R_i \Big|_{t = 0} = R_i^0 \in SO(n), \quad \Omega_i \in \mathfrak{so}(n), \quad ~ i \in [N].
\end{cases}
\end{equation}
Note that a global well-posedness of solutions to \eqref{B-5} can be guaranteed by  
the Cauchy--Lipschitz theorem, since $SO(n)$ is compact and the 
right-hand side is Lipschitz continuous; see the discussion below~\eqref{A-2}. \newline

Now, we define two modes of collective behavior that we study in the following sections.
\begin{definition}\label{D2.3}
Let $\mathbf{R} = \{ R_i \}$ be a  global solution to \eqref{B-5}.
\begin{enumerate}
\item $\mathbf{R}$ relaxes to equilibrium if there 
exist $R_i^\infty \in SO(n)$ such that
\[
\lim_{t \to \infty} R_i(t) = R_i^\infty 
\quad \text{and} \quad 
\lim_{t \to \infty} \dot{R}_i(t) = 0, 
\quad \forall\, i \in [N].
\]
\item $\mathbf{R}$ exhibits complete phase 
synchronization if
\[
\lim_{t \to \infty} \|R_i(t) - R_j(t)\| = 0, 
\quad \forall\, i, j \in [N].
\]
\end{enumerate}
\end{definition}
Below, we show that the Cauchy problem \eqref{B-5} is invariant under right multiplication of all 
oscillators by a fixed element of $SO(n)$, which allows us to choose the 
attraction point at $I_n$ without loss of generality. 
For the fixed attraction point $Q$, let $\tilde Q \in SO(n)$ be any 
other attraction point. Then, we set 
\[
\tilde\eta(\cdot\, Q^T\tilde Q) := \eta(\cdot), 
\qquad 
\tilde I(\cdot\, Q^T\tilde Q) := I(\cdot).
\]
\begin{lemma}[Rotational invariance]\label{L2.2}
Let $\{R_i\}_{i=1}^N$ be a global solution of \eqref{B-5} with the attraction 
point $Q$ and system functions $(\eta, I)$. Then, for any other attraction point $\tilde Q \in SO(n)$, the state $\{ \tilde R_i(t) := R_i(t)Q^T\tilde Q \}$ solves \eqref{B-5} with attraction point $\tilde Q$:
\begin{equation*}
\begin{cases}
\displaystyle \dot{{\tilde R}}_i {\tilde R}_i^T = \Omega_i 
+ \Big(\frac{\kappa}{2N} {\tilde \eta}({\tilde R}_i)\sum_{j=1}^{N}{\tilde I}({\tilde R}_j)\Big)
({\tilde Q} {\tilde R}_i^T - {\tilde R}_i {\tilde Q}^T), \quad t > 0, \\[6pt]
{\tilde R}_i \Big|_{t = 0} = R_i^0 Q^T\tilde Q \in SO(n), \quad \Omega_i \in \mathfrak{so}(n), \quad ~~ i \in [N].
\end{cases}
\end{equation*}
\end{lemma}
\begin{proof}
We substitute $\tilde R_i(t) = R_i(t)Q^T\tilde Q$ directly into \eqref{B-5} with attraction point $\tilde Q$ to get the desired system. 
\end{proof}
\begin{remark} \label{R2.3}
By Lemma~\ref{L2.2}, we may set $Q = I_n$ without loss of 
generality. Following the convention of the Winfree sphere model 
\cite{park2021generalization}, we also set $\eta \equiv 1$. Under 
these normalizations, the Cauchy problem \eqref{B-5} becomes
\begin{equation*}
\begin{cases}
\displaystyle \dot{R}_i = \Omega_i R_i 
+ \frac{\kappa}{2N}\Big(\sum_{j=1}^{N} I(R_j)\Big)(I_n - R_i^2), 
\quad t >0, \\[6pt]
R_i(0) = R_i^0 \in SO(n), \quad \Omega_i \in \mathfrak{so}(n), \quad  i \in [N],
\end{cases}
\end{equation*}
which is the form to be studied in the rest of the paper.
\end{remark}
\subsection{Reduction to the classical Winfree model} \label{sec:2.3}
In this subsection, we show that the normalized model \eqref{A-2} with $n = 2$ 
is equivalent to the classical Winfree model \eqref{A-1} with sensitivity 
function $S(\theta) = -\sin\theta$. To see this, we set
\[ 
R_i(t) := \begin{pmatrix} 
\cos\theta_i(t) & -\sin\theta_i(t) \\ 
\sin\theta_i(t) & \cos\theta_i(t) 
\end{pmatrix} \quad t \geq 0, \qquad \Omega_i := \begin{pmatrix} 
0 & -\nu_i \\ \nu_i & 0 
\end{pmatrix}, \quad i \in [N],
\]
and we substitute these ansatz into \eqref{A-2} and compare the resulting relation to find 
\[
-{\dot \theta}_i \sin \theta_i = -\nu_i \sin \theta_i + \frac{\kappa}{2N}\Big(\sum_{j=1}^{N} {\hat I}(\theta_j)\Big) (2 \sin^2 \theta_i),
\]
\[
-{\dot \theta}_i \cos \theta_i = -\nu_i \cos \theta_i + \frac{\kappa}{2N}\Big(\sum_{j=1}^{N} {\hat I}(\theta_j)\Big) (2 \sin \theta_i\cos\theta_i),
\]
where ${\hat I}(\theta_i) = I(R_i)$. Since at least one of the $\sin\theta_i$ and $\cos\theta_i$ are nonzero, we simplify the equation above into the classical Winfree model with $S(\theta) = -\sin\theta$:
\begin{equation*}
\begin{cases}
\displaystyle \dot\theta_i = \nu_i - \frac{\kappa}{N}
\sum_{j=1}^{N}\hat I(\theta_j)\sin\theta_i, 
\quad t > 0, \\[6pt]
\theta_i(0) = \theta_i^0, \quad i \in [N].
\end{cases}
\end{equation*}

\subsection{Basic properties of $SO(n)$} \label{sec:2.4}
In this subsection, we study basic properties of the special orthogonal group. First, we derive the distance formula using the following metric:
\[g(X, Y) = \langle X, Y \rangle \coloneqq \frac 1 2 \tr(X^T Y), \quad \|X\| := \sqrt{\langle X, X\rangle }, \qquad X, Y\in T_R SO(n).\]
Recall that we use the notation $A\succeq B$ if $A-B$ is positive semidefinite, and $A \sim B$ if $A$ and $B$ are orthogonally similar, i.e., there exists a real orthogonal matrix $P$ such that
\[A = P^TBP.\]
Note that the relation $\sim$ is an equivalence relation. In the following lemma, we list several useful properties for $SO(n)$ to be used throughout the paper.
\begin{lemma}\label{L2.3}
Let $R \in SO(n)$ be given. Then, the following assertions hold. 
\vspace{0.1cm}
    \begin{itemize}
        \item[(1)] The eigenvalue of $R$ has the form $\{e^{\mathrm{i}\theta_1}, e^{-\mathrm{i}\theta_1}, \cdots ,e^{\mathrm{i}\theta_m},  e^{-\mathrm{i}\theta_m}, 1, \cdots, 1\}$ for some integer $m\in [0, n/2]$ and $\theta_j\in (0, \pi]$,  including repetition of eigenvalues. multiplicity of $-1=e^{\mathrm{i}\pi}$ is even.
        \vspace{0.2cm}
        \item[(2)] $R \sim \Lambda$ where $\Lambda$ has a form
        \[\Lambda = \begin{pmatrix}
            {\mathcal R}(\theta_1)   \\
             &  \ddots   \\
              &  &   {\mathcal R}(\theta_m) \\
              &  & & I_{n-2m}
        \end{pmatrix},\]
    where ${\mathcal R}(\theta)$ is the rotation matrix on the plane with angle $\theta$:
    \[ {\mathcal R}(\theta):= \begin{pmatrix}
        \cos\theta & -\sin\theta \\
        \sin\theta & \cos\theta
    \end{pmatrix}.\]
    \vspace{0.2cm}
    \item[(3)] With $\Lambda$ given in the second assertion (2), let $X$ be
    \[X= \begin{pmatrix}
        0 & -\theta_1\\
         \theta_1 & 0 \\
        & &  \ddots \\
        & & & & 0 & -\theta_m \\
        & & & & \theta_m & 0 \\
        & & & & & & O_{n-2m}
    \end{pmatrix},\]
    where $O_{n-2m}$ is $(n-2m)\times (n-2m)$ zero matrix. Then we have
    \[\Lambda = \exp (X).\]
    \item[(4)] For any given anti-symmetric matrix $\Omega$, there exists a integer $k\in[0, n/2]$ and positive constant $\lambda_1,\cdots,\lambda_k$ such that $\Omega$ is orthogonally equivalent to the following block matrix.
    \[\Omega \sim \tilde \Omega = \text{diag}(\lambda_1 J, \lambda _2 J, \cdots, \lambda_k J, O_{n-2k}),\]
    where
    \[J = \begin{pmatrix}
        0 & -1 \\ 1 & 0
    \end{pmatrix}.\]
  \item[(5)] For the orthogonal real matrix $P$ such that $A = P^T\exp (X)P,$  let $\gamma:[0, 1]\rightarrow SO(n)$ be a smooth curve on $SO(n)$ defined as $\gamma(t) = P^T\exp(tX)P.$  Then $\gamma(t)$ is the length-minimizing geodesic connecting 
    \[ \gamma(0)=I_n \quad \mbox{and} \quad \gamma(1)=R. \]
    Moreover, the geodesic distance $d(I_n, R)$, which is the infimum of distance of every curve connecting $I_n$ and $R$, is
    \begin{equation} \label{B-5-1}
    d(I_n, R) = \sqrt{\sum_{i=1}^{m} \theta_i^2},
    \end{equation}
	and for $R, S\in SO(n)$, the geodesic distance between $R$ and $S$ is equal to $d(I_n, R^TS).$
	\[d(R, S) = d(I_n, R^TS).\]    
      \item[(6)] For positive semidefinite matrices $A,\ B,\ C$, if $B\succeq C$, then $\tr(AB)\ge \tr(AC)$.
    \end{itemize}
\end{lemma}
\begin{proof}
Although the proofs for the above facts may be found in books on special orthogonal matrix, we provide their proofs in Appendix \ref{Ap-A-1} for readers' convenience.
\end{proof}

\vspace{0.5cm}

Below, we discuss several structural assumptions on the influence function $I$:
\begin{itemize}
\item[(1)] $I(R)$ only depends on the Riemannian distance between $I_n$ and $R$:~there exists a function
$\tilde I : [0, \lfloor \frac n2 \rfloor \pi] \to \mathbb{R}_{\ge 0}$ such that
\[
I(R) = \tilde I\bigl(d(I_n, R)\bigr).
\]
\item[(2)]
    $\tilde I$ is decreasing and satisfies
    \[
    \lim_{r \to 0^+} \tilde I(r) = 1.
    \]
    This reflects the physical assumption that the influence of each oscillator
    becomes stronger as it approaches the attraction point.
    \vspace{0.1cm}
    \item[(3)]
    $\tilde I$ is globally Lipschitz continuous with compact support:
    \begin{equation} \label{B-6}
    \operatorname{supp} \tilde I = \bigl[0, \beta\bigr], \quad \mbox{  for some $0 < \beta < \frac{\pi}{2}$.}
    \end{equation}
\end{itemize}

\vspace{0.5cm}

Finally, we discuss the existence and uniqueness of the global solution to system $\eqref{A-2}$.

\begin{lemma} For given $R_i^0\in SO(n)$, there exist a unique solution $\{R_i(t)\}$ to \eqref{A-2} with initial data $\{R_i^0\}$, which satisfies
\[R_i(t) \in SO(n), \quad \forall\, t\ge0.\]
	
\end{lemma}
\begin{proof}
Note that by construction of dynamics, we have that the solution $\mathbf{R}(t)= \{ R_i(t) \}$ to~\eqref{A-2} is on $SO(n)^N$ until the life span.
For the local existence, we show Lipschitz continuity of right-hand-side of $\eqref{A-2}_1$ with respect to our norm defined at \eqref{A-3}. 

We rewrite the system $\eqref{A-2}$ as a dynamics of $\mathbf{R} = (R_1, \dots, R_N)$:
\[\dot{\mathbf{R}} = F(\mathbf{R})\]
for $F:(SO(n))^N \rightarrow (\bbr^{n\times n})^N.$
Then we have
\begin{align*}
\|F(\mathbf{R})-F(\mathbf{S})\|_{\ell^1 (\bbr^{n\times n})} &=\sum_{i=1}^N\left\|\Omega_i(R_i-S_i)+\frac{\kappa}{2N}\sum_{j=1}^N \big[I(R_j)(I_n-R_i^2)-I(S_j)(I_n-S_i^2)\big]\right\|\\
&\le \sum_{i=1}^N\underbrace{\|\Omega_i (R_i-S_i)\|}_{\eqqcolon \mathcal{I}_{11}}+\frac{\kappa}{2N}\sum_{i,j=1}^N \underbrace{\big\|I(R_j)(I_n-R_i^2)-I(S_j)(I_n-S_i^2)\big\|}_{\eqqcolon \mathcal{I}_{12}}
\end{align*}
From the basic properties of Frobenius norm, we have
\begin{equation}\label{B-7}
\begin{aligned}
&\|AB\|_F \le \|A\|_F\|B\|_F \quad \Longrightarrow \quad \|AB\|\le \sqrt2 \|A\|\|B\|, \quad A, B\in \bbr^{n\times n}.
\\
&\|RA\| = \|AR\| = \|A\|, \quad\|R\| = \|I_n\|= \sqrt{\frac n2}, \quad \forall\, R\in SO(n),\ A\in \bbr^{n\times n}.
\end{aligned}	
\end{equation}
Then, $\eqref{B-7}_1$ yields
\[\mathcal{I}_{11} \le \sqrt2 \|\Omega_i\|\|R_i-S_i\|.\]
For $\mathcal{I}_{12}$, we have 
\begin{align*}
&\quad\big\|I(R_j)(I_n-R_i^2)-I(S_j)(I_n-S_i^2)\big\| \\
&\hspace{1.5cm} \le\big\|(I(R_j)-I(S_j))(I_n-R_i^2)\big\|+\big\|I(S_j)(R_i^2-S_i^2)\big\|\\
&\hspace{1.5cm} \le \text{Lip}^*I\|R_j-S_j\| (\|I_n\| + \|R_i^2\|) \\
&\hspace{2cm}+ |I(S_j)||(\|R_i(R_i-S_i)\| + \|(R_i - S_i)S_i\|) \\
&\hspace{1.5cm} \le (\sqrt{2n}\, \text{Lip*}I + 2)\|\mathbf{R}-\mathbf{S}\|_{\ell^1(\bbr^{n\times n})},
\end{align*}
where Lip$^*I$ is a global Lipschitz constant of $I$.\[
\text{Lip*} I := \sup_{\substack{X,\, Y \in SO(n),\\ \, X \neq Y}} \frac{|I(X) - I(Y)|}{\|X - Y\|} .
\]
Combining above two inequalities give Lipschitz bound of $F$, thus by the standard Cauchy-Lipschitz theory in a finite dimensional manifold $(SO(n), \|\cdot\|)$, a local existence and uniqueness of the solution is proved.
To show global existence, one needs to control the magnitude of $F$.
This is straightforward:
\[
\|F(\mathbf{R})\|_{\ell^1(\bbr^{n\times n})}\le \sum_{i=1}^{N}\left(\|\Omega_i R_i\|+\Big(\frac{\kappa}{2N}
\sum_{j=1}^{N} |I(R_j)|\Big)(\|I_n\| +\| R_i^2\|)\right)\le 
\sum_{i=1}^{N}\sqrt{\frac n2}( \|\Omega_i\|+\kappa) 
\]

Now the only step left is to show that Lip$^*I<\infty$ provided that $\tilde I$ is Lipschitz. Let $B_r$ and $\bar B_r$ be the open ball of radius $r$ on $SO(n)$ and its closure:
\[B_r \coloneqq \{R : d(I_n, R)<r\},\quad \overline{B_r} \coloneqq \{R : d(I_n, R)\le r\}.\]
Then by the definition of Lip$^*I$,
\begin{align}
\begin{aligned} \label{B-9}
    \text{Lip*} I &=\sup_{\substack{X,\, Y \in SO(n),\\ \, X \neq Y}} \frac{|I(X) - I(Y)|}{\|X - Y\|} = \sup_{\substack{X,\, Y \in B_\beta,\\ \, X \neq Y}} \frac{|I(X) - I(Y)|}{\|X - Y\|} \\
    &= \sup_{\substack{X,\, Y \in B_\beta,\\ \, X \neq Y}} \left(\frac{|\tilde{I}(d(I_n, X)) - \tilde{I}(d(I_n, Y))|}{|d(I_n, X) - d(I_n, Y)|} \cdot \frac{|d(I_n, X) - d(I_n, Y)|}{\|X - Y\|}\right),
 \end{aligned}
\end{align}
where we have used the fact that $\tilde I$ has a support $[0,\beta]$.
Next, we estimate the second factor in \eqref{B-9}. We use (5) of Lemma~\ref{L2.3} to see
\begin{equation}\label{B-10}
\frac{|d(I_n, X) - d(I_n, Y)|}{\|X - Y\|}\le \frac{d(X, Y)}{\|X - Y\|}=
\frac{d(I_n, X^T Y)}{\|I_n - X^T Y\|}.
\end{equation}
Let $\{e^{\mathrm{i}\theta_1}, e^{-\mathrm{i}\theta_1}, \cdots, e^{\mathrm{i}\theta_m}, e^{-\mathrm{i}\theta_m}, 1, \cdots, 1\}$ be the eigenvalues of $X^T Y$, which is defined in (2) of Lemma \ref{L2.3}. 
Then since $X, Y\in B_\beta$, $d(X, Y)<2\beta$, we have
\begin{equation}\label{B-11}
d(X, Y) = d(I_n, X^T Y) = \sqrt{\sum_{i=1}^{m}\theta_i^2}, \quad   0<\theta_i/2\le\beta.
\end{equation}
Now by (1) of Lemma \ref{L2.3}, we have
\begin{equation}\label{B-12}
\begin{aligned}
	\|I_n - X^{-1}Y\|^2 &= \frac12 \tr \Big((I_n-Y^T X)(I_n -X^T Y)\Big) \\
	&=\frac12 \tr(I_n-Y^T X-X^TY + Y^T X X^T Y) 
	\\
    &=\tr (I_n - X^TY)
	=\sum_{i=1}^{m}(2-2\cos\theta_i) =\sum_{i=1}^{m}4\sin^2 \frac{\theta_i}{2}.
\end{aligned}
\end{equation}
We plug $\eqref{B-11}$ and $\eqref{B-12}$ into \eqref{B-10} to obtain
\begin{equation}\notag\label{B-13}
\frac{|d(I_n, X) - d(I_n, Y)|}{\|X - Y\|}\le \frac{|d(X, Y)|}{\|X - Y\|}=
\frac{|d(I_n, X^{-1}Y)|}{\|I_n - X^{-1}Y\|}=\frac{\sqrt{\sum_{i=1}^m\theta_i^2}}{2\sqrt{\sum_{i=1}^m \sin^2\frac{\theta_i}{2}}}\le \frac{\beta}{\sin\beta}.
\end{equation}
Note that, in the last inequality, we used the fact that $0\le \theta_i/2\le \beta$ and 
\[
0\le x\le \alpha <\pi\implies \frac{\sin \alpha}{\alpha} x\le \sin x\le x.
\]
Therefore, we have
\[
\text{Lip*} I\le \frac{\beta}{\sin\beta}\sup_{\substack{X,\, Y \in SO(n),\\ \, X \neq Y}}  \frac{|\tilde{I}(d(I_n, X)) - \tilde{I}(d(I_n, Y))|}{|d(I_n, X) - d(I_n, Y)|} \le \frac{\beta}{\sin\beta}\text{Lip}\tilde{I}.
\]
\end{proof}

\section{Uniform-in-time stability}\label{sec:3}
\setcounter{equation}{0}
In this section, we study the emergent dynamics of the generalized Winfree model \eqref{A-2} on $SO(n)$. Thanks to Lemma \ref{L2.2}, we may set the attraction $Q$ to be $I_n$.
\subsection{Existence of a trapping region} \label{sec:3.1}
The main hurdle for the analysis of this model is that we can not directly compute the dynamics of the distance function $d(I_n, R)$. We instead show that the quantity $n-\tr(R)$ is equivalent with $d(I_n, R)^2$, and then we estimate $n-\tr(R)$. 
\begin{lemma} \label{L3.1}
Let $\mathbf R=\{R_i\}_{i=1}^N$ be a global solution to \eqref{A-2} satisfying a priori condition:
\[  \sup_{0 \leq t < \infty} d(I_n, R_i(t)) \leq 2 \pi, \quad \forall~i \in [N]. \]
Then, the following relations hold:~For $i \in [N]$ and $t \geq 0$, 
\begin{equation} \label{D-1}
4\sin^2 \Big( \frac{d(I_n, R_i(t))}{2}  \Big) \le n-\tr(R_i(t)) \le d(I_n, R_i(t))^2.
\end{equation}
\end{lemma}
\begin{proof}
For each $R_i \in SO(n)$, (1), (2) and (5) in Lemma \ref{L2.3} yield that the eigenvalue of $R_i$ has the form:
\[ \{e^{\mathrm{i}\theta_{i1}}, e^{-\mathrm{i}\theta_{i1}}, \cdots ,e^{\mathrm{i}\theta_{im_i}},  e^{-\mathrm{i}\theta_{im_i}}, 1, \cdots, 1\} \]
for some integer $m_i \in [0, n/2]$ and $\theta_{ij} \in (0, \pi]$ and 
\begin{equation} \label{D-1-0-0}
\tr(R_i) = n - 2m_i + 2 \sum_{j=1}^{m_i} \cos \theta_{ij} \quad \mbox{and} \quad     d(I_n, R_i)^2 = \sum_{j=1}^{m_i} \theta_{ij}^2. 
\end{equation}
These imply
\begin{align}
\begin{aligned} \label{D-1-0}
n-\tr(R_i) &= n-\Big(n-2m_i + \sum_{j=1}^{m_i}2\cos\theta_{ij}\Big) = 2 \sum_{j=1}^{m_i} (1- \cos \theta_{ij})  \\
&= \sum_{j=1}^{m_i} 4 \sin^2 \Big(\frac{\theta_{ij}}{2}\Big).
\end{aligned}
\end{align}
On the other hand, we recall basic trigonometric inequalities:
\begin{equation} \label{D-1-0-1}
0 <\theta < \alpha \le \pi \quad \Longrightarrow\quad \frac{\sin^2\alpha}{\alpha^2} \theta^2 < \sin^2\theta < \theta^2 < \alpha^2.
\end{equation}
Since 
\[ \frac{\theta_{ij}}{2} \le \frac{d(I_n, R_i)}{2} \le \pi, \quad j \in [m_i], \]
we use \eqref{D-1-0-1} to find 
\begin{equation} \label{D-1-0-2}
4 \frac{\sin^2 \Big( \frac{d(I_n, R_i)}{2}\Big)}{d(I_n, R_i)^2} \theta_{ij}^2 \leq  4 \sin^2 \Big( \frac{\theta_{ij}}{2} \Big)
 \leq \theta_{ij}^2.
\end{equation}
We sum up the relations \eqref{D-1-0-2} over all $j \in [m_i]$ and use $\eqref{D-1-0-0}_2$ to find 
\[
4 \frac{\sin^2 \Big( \frac{d(I_n, R_i)}{2}\Big)}{d(I_n, R_i)^2}  \sum_{j =1}^{m_i} \theta_{ij}^2 \leq  \sum_{j=1}^{m_i} 4 \sin^2 \Big( \frac{\theta_{ij}}{2} \Big)
 \leq \sum_{j = 1}^{m_i} \theta_{ij}^2.
\]
Now, we use  $\eqref{D-1-0-0}_2$ and \eqref{D-1-0} to get the desired relations:
\[
4 \sin^2 \Big( \frac{d(I_n, R_i)}{2}\Big) \leq n-\tr(R_i) \leq d(I_n, R_i)^2.
\]
\end{proof}
\begin{remark} Below, we provide two comments on Lemma \ref{L3.1}.
\begin{enumerate}
\item
The estimate \eqref{D-1} can be rewritten as follows.
\[ 2\sin \Big( \frac{d(I_n, R_i(t))}{2}  \Big) \le \sqrt{n-\tr(R_i(t))} \le d(I_n, R_i(t)), \]
or equivalently
\begin{equation} \label{D-1-0-3}
\sqrt{n-\tr(R_i(t))} \leq d(I_n, R_i) \leq 2 \arcsin \Big(\frac{\sqrt{n-\tr(R_i(t))}}{2} \Big).
\end{equation}
\item
One might wonder why we need an equivalence relation between $d(I_d, R_i)$ and $\sqrt{n -\tr(R_i)}$. By \eqref{B-5-1}, if we estimate the time-evolution of $d(I_n, R_i)$ directly, i.e.,
\[ \frac{d}{dt} d(I_n, R_i) = \frac{\sum_{j=1}^{m_i} \theta_{ij} {\dot \theta}_{ij}}{\sqrt{\sum_{j=1}^{m_i} \theta_{ij}^2}}. \]
and one can proceed using the evolution of $\theta_{ij}$ as is done in our companion paper \cite{ha2026convergence}. We instead use the dynamics for $n - \tr(R_i)$ and then using the equivalence relation \ref{D-1}, we can derive an estimate for $d(I_n, R_i)$ as can be seen in what follows.
\end{enumerate}
\end{remark}

Now, we introduce a sufficient framework $({\mathcal F}_A)$ in terms of initial data and free and system parameters:
 \vspace{0.1cm}
 \begin{itemize}
 \item
 $({\mathcal F}_A1)$:~Parameters $\beta, \gamma_0$ and $\gamma$ satisfy 
 \[
0 <  \beta  < \frac{\pi}2, \quad  0 < \gamma_0 < 2\sin \Big (\frac{\beta}{2} \Big), \quad  \gamma = 2\arcsin \Big (\frac{\gamma_0}{2} \Big ) < \beta.
 \]
 Recall that $\beta$ is the upper bound of the support of $\tilde I$ (see \eqref{B-6}).
 \vspace{0.2cm}
 \item
 $({\mathcal F}_A2)$:~Coupling strength is sufficiently large such that 
 \[  \kappa  > \frac{\max_{j \in [N]} \|\Omega_j\|}{\tilde{I}(\gamma)\sin \Big(2\sin \frac{\gamma}{2} \Big )}. \]
 \item
  $({\mathcal F}_A3)$:~Initial data is confined in a $(B_{\gamma_0})^N$:
  \[
   R_i^0 \in B_{\gamma_0}, \quad \forall~i\in [N].
  \]
  \end{itemize}   
 Next, we show that under the framework $(\mathcal F_A)$, every trajectory starting from $B_{\gamma_0}$ is confined in $\overline{ B_\gamma}$ for suitable $\gamma$. 
\begin{proposition}\label{P3.1}
Suppose that initial data and system parameters satisfy the framework $({\mathcal F}_A)$, and let $\mathbf R=\{R_i\}_{i=1}^N$ be a global solution to \eqref{A-2}.  Then we have 
\[ R_i(t)\in \overline{B_\gamma} \quad \mbox{for all $t\ge 0$}. \]
\end{proposition}
\begin{proof} 
We claim that 
\begin{equation}  \label{D-2-0}
n-\tr(R_i) \le 4\sin^2 \frac{\gamma}{2},  \quad \forall\ t\ge 0,\ i\in [N].
\end{equation}
Once we verify \eqref{D-2-0}, we use \eqref{D-1} and \eqref{D-1-0} that 
\begin{equation} \label{D-2-0-0}
4\sin^2 \Big( \frac{d(I_n, R_i)}{2}  \Big) \leq 4\sin^2 \frac{\gamma}{2},  \quad \forall\ t\ge 0,\ i\in [N].
\end{equation}
Since $\frac{\gamma}{2} < \frac{\pi}{2}$, the relation \eqref{D-2-0-0} yields the desired estimate:
\[
d(I_n, R_i(t)) \leq \gamma, \quad \mbox{i.e.,} \quad R_i(t) \in \overline{B_{\gamma}}, \quad t \geq 0,~~i \in [N].
\]
Indeed, $R_i^0$ initially sits in $B_{\gamma_0}\subset B_\gamma$, and since the $\sin$ function is increasing on $[0,\gamma/2]\subseteq[0,\pi/2]$, $R_i$ cannot escape $\overline{B_\gamma}$. \newline

\noindent {\it Proof of \eqref{D-2-0}}:~Suppose the contrary holds, i.e., then there exists the first exit time from the region $B_\gamma$:
\[
t_e \coloneqq \inf\Big\{t \in (0, \infty)\, | \, n-\tr(R_i(t)) > 4\sin^2 \frac{\gamma}{2} \ \text{ for some } i\in[N] \Big\}<\infty.
\]
Now, we set
\[
M(t) := \max_{i\in[N]}(n-\tr(R_i(t))).
\]
Then we have
\begin{equation}\label{D-2}
n-\tr(R_i(t)) \le 4\sin^2 \frac{\gamma}{2} \quad \Longrightarrow \quad d(I_n, R_i) \le \gamma, \quad \forall\ t\in[0, t_e], \ i\in[N].
\end{equation}
Moreover, it follows from the continuity of $n-\tr(R_i)$ and the definition of $t_e$ that 
\begin{equation}\label{D-3}
M(t_e) = 4\sin^2 \frac{\gamma}{2}, \quad D^+ M(t_e)\ge 0,
\end{equation}
where $D^+ M$ is the upper Dini derivative of $M(t)$.  \newline

\noindent We use the equation of $R_i$ gives
\begin{align} \label{D-4}
\begin{aligned}
\frac{d}{dt} (n-\tr(R_i)) &= -\tr (\dot R_i) = -\tr(\Omega_i R_i) - \frac{\kappa}{2N}\sum_{j=1}^{N} I(R_j)(n - \tr(R_i^2))  \\
&=: {\mathcal I}_{11} + {\mathcal I}_{12}.
\end{aligned}
\end{align}
Now, we estimate the terms in the right-hand side of \eqref{D-4} one by one. \newline

\noindent $\bullet$~Case A (Estimate of ${\mathcal I}_{11}$):~Note that
\begin{align}\label{D-5}
\begin{aligned}
{\mathcal I}_{11} &= -\tr(\Omega_i R_i) =-\tr(\Omega_i^T R_i^T) = \tr\left(\Omega_i^T \frac{R_i-R_i^T}{2}\right) = \langle \Omega_i, R_i-R_i^T\rangle \\
&\le  \|\Omega_i\|\cdot\|R_i-R_i^T\| = \|\Omega_i\|\sqrt{\frac12\,\tr\Big((R_i^T - R_i)(R_i - R_i^T)\Big)} \\
&= \|\Omega_i\|\sqrt{n - \tr(R_i^2)} = 2\|\Omega_i\|\sqrt{\sum_{j=1}^{m_i} \sin^2 \theta_{ij}},
\end{aligned}
\end{align}
where we used
\[
\tr(R_i^2) = n-2m_i+2\sum_{j=1}^{m_i} \cos 2\theta_{ij} = n-4\sum_{j=1}^{m_i} \sin^2 \theta_{ij}
\]
for theh last equality.

\vspace{0.2cm}

\noindent $\bullet$~Case B (Estimate of ${\mathcal I}_{12}$):
By above equality, we have
\begin{equation}\label{D-6}
{\mathcal I}_{12} = - \frac{2\kappa}{N}\sum_{j=1}^{N}  \tilde{I}(d(I_n,R_j))  \sum_{j=1}^{m_i} \sin^2 \theta_{ij}.
\end{equation}
Now, we combine \eqref{D-4}, \eqref{D-5}, \eqref{D-6} with the decreasing property of $\tilde I$ to obtain
\begin{equation*}\label{D-7}
\begin{aligned}
\frac{d}{dt} (n-\tr (R_i)) &\le 2\|\Omega_i\|\sqrt{\sum_{j=1}^{m_i} \sin^2 \theta_{ij}} 
-\frac{2\kappa}{N}\sum_{j=1}^{N} \tilde{I}(d(I_n,R_j))\cdot\sum_{j=1}^{m_i} \sin^2 \theta_{ij} \\
& \le 2\sqrt{\sum_{j=1}^{m_i} \sin^2 \theta_{ij}} \Bigg(\|\Omega_i\| - \kappa \tilde{I}(\max_{l\in [N]}d(I_n, R_l)) \sqrt{\sum_{j=1}^{m_i} \sin^2 \theta_{ij}}\Bigg).
\end{aligned}
\end{equation*}
Note that $\max_{j\in [m_i]}|\theta_{ij}|\le d(I_n, R_j)\le \gamma<\frac{\pi}{2}$ implies
\[
\sqrt{\sum_{j=1}^{m_i}\sin^2 \theta_{ij}} \ge \frac{\sin d(I_n, R_i)}{d(I_n, R_i)}\sqrt{\sum_{j=1}^{m_j}\theta_{ij}^2} = \sin d(I_n, R_i).
\]
This yields
\begin{equation}\label{D-8}
\frac{d}{dt} (n-\tr(R_i)) \le 2\sqrt{\sum_{j=1}^{m_i}\sin^2 \theta_{ij}} \Big(\|\Omega_i\| - \kappa \tilde{I}(\max_{l\in [N]}d(I_n, R_l)) \sin d(I_n, R_i)\Big).
\end{equation}
Now, we set $\mathcal{I}(t_e)$ to be the set of maximal indices at time $t_e$:
\[
\mathcal{I}(t_e) \coloneqq \{ i\in[N]\, | \, n-\tr(R_{i}(t_e))=M(t_e)\}.
\]
Then for any $i^*\in \mathcal{I}(t_e)$, we use \eqref{D-1}, \eqref{D-2}, \eqref{D-8} and the condition on $\gamma$ to see
\begin{equation}\label{D-9}
\begin{aligned}
\frac{d}{dt} (n-\tr(R_{i^*})) \Big|_{t =t_e} &\le 2\sqrt{\sum_{j=1}^{m_{i^*}}\sin^2 \theta_{i^*j}} \Big(\|\Omega_{i^*}\| - \kappa \tilde{I}(\max_{l\in [N]}d(I_n, R_l)) \sin d(I_n, R_{i^*})\Big) \\
&\le 2\sqrt{\sum_{j=1}^{m_{i^*}}\sin^2 \theta_{i^*j}} \Big(\|\Omega_{i^*}\| - \kappa \tilde{I}(\gamma) \sin \sqrt{n-\tr(R_{i^*})}\Big) \\
&\le 2\sqrt{\sum_{j=1}^{m_{i^*}}\sin^2 \theta_{i^*j}} \Big(\max_{j\in [N]}\|\Omega_j\| - \kappa \tilde{I}(\gamma)\sin (2\sin \frac{\gamma}{2})\Big) \\
&< 0.
\end{aligned}
\end{equation}
This implies
\[
D^+M(t_e) = \max_{i^*\in \mathcal{I}_t} \frac{d}{dt} (n-\tr(R_{i^*})) < 0.
\]
This is contradictory to \eqref{D-3}. Therefore, we have $t_e=\infty$, and we obtain the desired result.
\end{proof}
\begin{remark} By the condition on $\gamma$ in $({\mathcal F}_A1)$, one has
\[
\frac{\gamma_0}{2}=\sin\frac{\gamma}{2}<\frac{\gamma}{2}.
\]

Thus, we have
\[ B_{\gamma_0} \subset B_{\gamma}. \]
\end{remark}

In the following lemma, we can control a specific trajectory by requiring a larger coupling strength compared to that appearing in $({\mathcal F}_A2)$.
\begin{lemma}\label{L3.2} 
Suppose that initial data and system parameters satisfy $({\mathcal F}_A1)$ and additional conditions:~there exists an index $i_* \in [N]$ such that 
\begin{equation}\notag \label{D-9-0-0}
R_{i_*}^0 \in B_{\gamma_0}, \quad \kappa >  \frac{N \|\Omega_{i_*} \|}{  \tilde{I}(\gamma) \sin(2\sin \frac{\gamma}{2})}, 
\end{equation}
and let $\{R_i(t)\}$ be a solution to \eqref{A-2}. Then we have 
\[  R_{i_*}(t) \in \overline{B_{\gamma}}, \quad \forall~t \geq 0. \]
\end{lemma}
\begin{proof} By the same argument (see \eqref{D-1-0}) as in the proof of Proposition \ref{P3.1}, it suffices to show that if there exists a finite time $t_e \geq 0$ such that 
\[ n-\tr(R_{i_*}) \Big|_{t = t_e} = 4\sin^2 \frac{\gamma}{2}, \]
then we have
\begin{equation} \label{D-9-0} 
\frac{d}{dt}  \Big|_{t = t_e} (n-\tr(R_
{i_*})) \le 0. 
\end{equation}
To see \eqref{D-9-0}, we assume that there exits a finite time $t_e \geq 0$ such that 
\[  n-\tr(R_{i_*}(t_e)) = 4\sin^2 \frac{\gamma}{2}. \]
Then, it follows from \eqref{D-1} that 
\[
4\sin^2 \Big( \frac{d(I_n, R_{i_*}(t_e))}{2}  \Big) \le 4\sin^2 \frac{\gamma}{2} \le d(I_n, R_{i_*}(t_e))^2,
\]
i.e., we have
\[ 2\sin \frac{\gamma}{2} \le d(I_n, R_{i_*}(t_e)) \le \gamma. \]
 Now, we use this and the positivity of $\tilde I$ to see
\[\sum_{j=1}^{N} \tilde I(d(I_n, R_j(t_e))) \ge \tilde I(d(I_n, R_{i_*}(t_e))) \ge \tilde I(\gamma). \] 
If we use the same procedure as in \eqref{D-9}, then we get the desired estimate:
\begin{align}
\begin{aligned}\notag  \label{D-9-1}
\frac{d}{dt} (n-\tr(R_{i_*})) &\le 2\|\Omega_{i_*} \|\sqrt{\sum_{j=1}^{m_{i_*}} \sin^2 \theta_{lj}} 
- \frac{2\kappa}{N}\sum_{j=1}^N \tilde{I}(d(I_n, R_j))\cdot\sum_{j=1}^{m_{i_*}} \sin^2 \theta_{i_* j} \\
&\le 2\sqrt{\sum_{j=1}^{m_{i_*}} \sin^2 \theta_{i_*j}}  \underbrace{\Big(\|\Omega_{i_*}\| - \frac{\kappa}{N} \tilde{I} (\gamma) \sin(2\sin \frac{\gamma}{2})\Big)}_{< 0, \quad \mbox{by assumption on $\kappa$}} \\
&\leq 0.
\end{aligned}
\end{align}
Here all quantities are evaluated at $t = t_e$.  Therefore, we have the desired estimate:
\[  R_{i_*}(t) \in \overline{B_{\gamma}}, \quad \forall~t \geq 0. \]  
\end{proof}
\begin{lemma} \label{L3.3}
Let $x$ be a value in $[0, 1]$. Then, the following trigonometric identity holds.
\[ 2 \cos \Big(\frac{1}{2} \arcsin x \Big) = \sqrt{2 + 2 \sqrt{1-x^2}}. \]
\end{lemma}
\begin{proof} We set 
\[ \theta:= \frac{1}{2} \arcsin x ,\quad X:=  2 \cos \Big(\frac{1}{2} \arcsin x \Big) = 2 \cos \theta. \]
These and trigonometry identity yield
\[ x = \sin (2 \theta) = 2 \sin  \theta \cos \theta = 2 \sqrt{1- \cos^2 \theta} \cos \theta = 2 \sqrt{1 - \Big(  \frac{X}{2} \Big)^2} \frac{X}{2}, \]
i.e.,
\[ x = 2 \sqrt{1 - \Big(  \frac{X}{2} \Big)^2} \frac{X}{2}. \]
We square the above relation to rewrite $X$ in terms of $x$:
\[
x^2 = 4  \Big[ 1 - \Big(  \frac{X}{2} \Big)^2 \Big] \frac{X^2}{4} =  (4 - X^2) \frac{X^2}{4}.
\]
This yields
\[ X = \sqrt{2 + 2 \sqrt{1- x^2}}. \]
\end{proof}
Next, we are ready to provide our first main result. If at least one particle stays in $B_\gamma$ and $\kappa$ is strong enough, then all the oscillator eventually enters some ball.  For this, we set 
\[ \Gamma_i(\gamma) = \Gamma_i(N, \kappa, \Omega_i, \gamma) :=  \sqrt{2+2\sqrt{1-\frac{N^2 \|\Omega_i\|^2}{\kappa^2 \tilde I(\gamma)^2}}}.      \]
\begin{proposition}\label{P3.2} 
Suppose that the initial data and system parameters satisfy $({\mathcal F}_A1)$ and additional conditions:~there exists an index $i_* \in [N]$ such that $R_{i_*}^0 \in B_{\gamma_0}$, and 
\begin{equation} \label{D-9-2}
R_i^0 \in 
B_{\Gamma_i(\gamma)}, \quad   \kappa >  \kappa_c(\gamma) :=  \frac{N\max_{j \in [N]} \|\Omega_j\|}{\sin(2\sin \frac{\gamma}{2}) \tilde I(\gamma)},  \quad  \forall~i \in [N],
\end{equation}
and let $\{R_i\}$ be a global solution to \eqref{A-2}. Then we have 
\begin{equation} \label{D-9-3}
\limsup_{t \to \infty} d(I_n,R_i(t)) \le 2\arcsin\Big[ \frac12 \arcsin\Big(\frac{N\|\Omega_i\|}{\kappa \tilde{I}(\gamma)}\Big)\Big ], \quad \forall\ i \in [N].
\end{equation}
\end{proposition}
\begin{proof}
By Lemma \ref{L3.2}, we have 
\begin{equation} \label{D-9-4}
d(I_n, R_{i_*}(t)) \le \gamma \quad \mbox{for all $t \ge 0$}.
\end{equation}
First, we control $\sqrt{n-\tr(R_i)}$ and then using the first inequality in \eqref{D-9-4}, we derive the  desired estimate \eqref{D-9-3}. \newline

\noindent $\bullet$~Step A (Derivation of the temporal evolution for $n-\tr(R_i)$):~For $i\in [N]$, we use the same argument as in \eqref{D-9} to see 
\begin{align}\label{D-10}
\begin{aligned}
\frac{d}{dt} (n-\tr(R_i)) &\le 2\|\Omega_i\|\sqrt{\sum_{j=1}^{m_i} \sin^2 \theta_{ij}} 
- \frac{2\kappa}{N}\sum_{j=1}^N \tilde{I}(d(I_n, R_j))\cdot\sum_{j=1}^{m_i} \sin^2 \theta_{ij} \\
&\le 2\sqrt{\sum_{j=1}^{m_i} \sin^2 \theta_{ij}} \Big(\|\Omega_i\| - \frac{\kappa}{N} \tilde{I}(\gamma) \sin d(I_n, R_i)\Big),
\end{aligned}
\end{align}
where we used the following relation: By \eqref{D-9-4}, 
\[
\sum_{j=1}^N \tilde{I}(d(I_n, R_j)) \geq  \tilde{I}(d(I_n, R_{i_*})) \geq \tilde{I}(\gamma).
\]

\vspace{0.2cm}

\noindent $\bullet$~Step B  (Decreasing mode of $n-\tr(R_i)$): It follows from \eqref{D-10}. that 
\begin{equation} \label{D-10-1}
\underbrace{\sin d(I_n, R_i) > \frac{N\|\Omega_i\|}{\kappa \tilde{I}(\gamma)}}_{=:\Delta}  \quad \Longrightarrow \quad \frac{d}{dt} (n-\tr(R_i)) \leq 0. 
\end{equation}
Next, we discuss condition to guarantee the condition $\Delta$ in terms of $n-\tr(R_i)$. By the graphical property of the map $x \mapsto \sin x$, one has 
\begin{align}
\begin{aligned} \label{D-10-3}
& \sin d(I_n, R_i) > \frac{N\|\Omega_i\|}{\kappa \tilde{I}(\gamma)} \\
& \hspace{1cm}  \Longleftrightarrow  \arcsin\Big(\frac{N\|\Omega_i\|}{\kappa \tilde{I}(\gamma)}\Big) < d(I_n, R_i)  < \pi - \arcsin\Big(\frac{N\|\Omega_i\|}{\kappa \tilde{I}(\gamma)}\Big).
\end{aligned}
\end{align}
On the other hand, we recall \eqref{D-1-0-3}:
\begin{equation} \label{D-10-3-1}
\sqrt{n-\tr(R_i(t))} \leq d(I_n, R_i) \leq 2 \arcsin \Big(\frac{\sqrt{n-\tr(R_i(t))}}{2} \Big).
 \end{equation}
 In order to find the condition for $\sqrt{n-\tr(R_i(t))}$ to satisfy the relation \eqref{D-10-3}, we require 
 \begin{align}
 \begin{aligned} \label{D-10-3-2}
&  \arcsin\Big(\frac{N\|\Omega_i\|}{\kappa \tilde{I}(\gamma)}\Big)  < \sqrt{n-\tr(R_i(t))} \quad \mbox{and} \\
& \hspace{1.5cm}  2 \arcsin \Big(\frac{\sqrt{n-\tr(R_i(t))}}{2} \Big) < \pi - \arcsin\Big(\frac{N\|\Omega_i\|}{\kappa \tilde{I}(\gamma)}\Big).
 \end{aligned}
 \end{align}
 In particular, the second relation \eqref{D-10-3-2} can be rewritten as 
\begin{align}
\begin{aligned} \label{D-10-3-3}
&   2 \arcsin \Big(\frac{\sqrt{n-\tr(R_i(t))}}{2} \Big) < \pi - \arcsin\Big(\frac{N\|\Omega_i\|}{\kappa \tilde{I}(\gamma)} \Big) \\
& \hspace{1cm} \Longleftrightarrow \quad   \sqrt{n-\tr(R_i)} < 2\sin\Big(\frac{\pi}{2} - \frac12 \arcsin\Big(\frac{N \|\Omega_i\|}{\kappa \tilde{I}(\gamma)}\Big)\Big).
\end{aligned}
\end{align}
Note that the upper bound in \eqref{D-10-3-3} can be further simplified using Lemma \ref{L3.3} as follows. 
\begin{align}
\begin{aligned} \label{D-10-3-4}
& 2\sin\Big(\frac{\pi}{2} - \frac12 \arcsin\Big(\frac{N \|\Omega_i\|}{\kappa \tilde{I}(\gamma)}\Big)\Big)  \\
& \hspace{1cm} = 2\cos\Big(\frac12 \arcsin\Big(\frac{N \|\Omega_i\|}{\kappa \tilde{I}(\gamma)}\Big)\Big) = \sqrt{2+2\sqrt{1-\frac{N^2 \|\Omega_i\|^2}{\kappa^2 \tilde I(\gamma)^2}}}  =  \Gamma(\gamma).
\end{aligned}
\end{align}
Finally, we combine \eqref{D-10-1},  \eqref{D-10-3},  \eqref{D-10-3-1}, \eqref{D-10-3-2}, \eqref{D-10-3-3}  and \eqref{D-10-3-4} to see that  $R_i$ satisfies the following implications:
\begin{align}
\begin{aligned}\notag \label{D-10-3-5}
&  \arcsin\Big(\frac{N\|\Omega_i\|}{\kappa \tilde{I}(\gamma)}\Big)  < \sqrt{n-\tr(R_i)} < \Gamma(\gamma) \\
& \hspace{1cm} \Longrightarrow \quad \sin d(I_n, R_i) > \frac{N\|\Omega_i\|}{\kappa \tilde{I}(\gamma)}  \quad \Longrightarrow \quad \frac{d}{dt} (n-\tr(R_i)) \leq 0.
\end{aligned}
\end{align}
\noindent $\bullet$~Step C (Derivation of asymptotic dynamics of $\sqrt{n-\tr(R_i(t))}$:  By the condition $\eqref{D-9-2}_2$, we have
\[ 
d(I_n, R_{i}^0) < \Gamma(\gamma).
\]
Below, we use a phase-line analysis and comparison principle of ODE related to \eqref{D-10}.  Since $\sqrt{n-\tr(R_i^0)}\le d(I_n,R_i^0)$, we have the following two cases:
\[ \mbox{Either} \quad  \arcsin\Big(\frac{N\|\Omega_i\|}{\kappa \tilde{I}(\gamma)}\Big)  <  \sqrt{n-\tr(R_i^0)} < \Gamma(\gamma) \quad \mbox{or} \quad  \sqrt{n-\tr(R_i^0)} \leq  \arcsin\Big(\frac{N\|\Omega_i\|}{\kappa \tilde{I}(\gamma)}\Big).
\]
\noindent $\diamond$~Case C.1:~Suppose that 
\[  \arcsin\Big(\frac{N\|\Omega_i\|}{\kappa \tilde{I}(\gamma)}\Big)  <  \sqrt{n-\tr(R_i^0)} < \Gamma(\gamma). \]
In this case $n-\tr(R_i)$ is in decreasing mode initially. Therefore, $n - \tr(R_i)$ begin to decrease unit it reach the value $ \arcsin\Big(\frac{N\|\Omega_i\|}{\kappa \tilde{I}(\gamma)}\Big)$. If it passes $ \arcsin\Big(\frac{N\|\Omega_i\|}{\kappa \tilde{I}(\gamma)}\Big)$, it reduces to Case C.2. \newline

\noindent $\diamond$~Case C.2:~Suppose that 
\[ \sqrt{n-\tr(R_i^0)} \leq  \arcsin\Big(\frac{N\|\Omega_i\|}{\kappa \tilde{I}(\gamma)}\Big). \]
Then, the differential inequality \eqref{D-10} does not yield the dynamic mode of $n - \tr(R_i)$. When it passes $ \arcsin\Big(\frac{N\|\Omega_i\|}{\kappa \tilde{I}(\gamma)}\Big)$ as time goes on, it begin to decrease from that instant. Hence it falls down to Case C.2 again.  

\vspace{0.2cm}

By Case C.1 and Case C.2, we have
\begin{equation} \label{D-10-3-6}
\limsup_{t \to \infty} \sqrt{n-\tr(R_i(t))} \le \arcsin\Big(\frac{N \|\Omega_i\|}{\kappa \tilde{I}(\gamma)}\Big).
\end{equation}

\noindent $\bullet$~Step D (Derivation of asymptotic dynamics of $R_i$):~We combine \eqref{D-10-3-6} and \eqref{D-1}:
\[ 
2\sin \Big( \frac{d(I_n, R_i(t))}{2}  \Big) \le  \sqrt{n-\tr(R_i(t))}
\]
to derive the desired estimate \eqref{D-9-3}.
\end{proof}
In what follows, we provide two direct applications of Proposition \ref{P3.2}.
\begin{corollary}\label{C3.1}
Suppose that initial data and system parameters satisfy $({\mathcal F}_A1)$ and additional conditions:~there exists an index $i_* \in [N]$ such that $R_{i_*}^0 \in B_{\gamma_0}$, and
\begin{equation} \label{D-10-4}
\kappa >  \kappa_{c}(\gamma, 1) :=  \frac{N\max_{j \in [N]} \|\Omega_j\|}{\sin(2\sin \frac{\gamma}{2}) \tilde I(\gamma)},\quad R_i^0 \in 
B_{\Gamma_i(\gamma)}, \quad \forall~i \in [N], 
\end{equation}
and let $\{R_i\}$ be a global solution to \eqref{A-2}. Then we have
\begin{equation}\notag\label{D-12}
    \limsup_{t \to \infty} d(I_n,R_i(t)) < 2\arcsin\Big[ \frac12 \arcsin\Big( \frac{\kappa_c(\gamma)}{\kappa}\sin(2\sin(\frac{\gamma}{2}))\Big)\Big ] < \gamma, \quad \forall\ i \in [N].
\end{equation}
\end{corollary}
\begin{proof} Since our setting is exactly the same as in Proposition \ref{P3.2}, we have
\begin{equation} \label{D-12-0}
\limsup_{t \to \infty} d(I_n,R_i(t)) \le 2\arcsin\Big[ \frac12 \arcsin\Big(\frac{N\|\Omega_i\|}{\kappa \tilde{I}(\gamma)}\Big)\Big ], \quad \forall\ i \in [N]. 
\end{equation}
On the other hand, the relation $\eqref{D-9-2}_1$ and the assumption ${\kappa_c}/{\kappa} < 1$ imply
\begin{equation} \label{D-12-1}
\frac{N\|\Omega_j\|}{\kappa\tilde{I}(\gamma)} \le \frac{N\max_{j\in[N]}\|\Omega_j\|}{\kappa\tilde I(\gamma)} = \frac{\kappa_c(\gamma)}{\kappa} \sin \Big (2\sin(\frac{\gamma}{2}) \Big ) <  \sin \Big (2\sin(\frac{\gamma}{2}) \Big ).  
\end{equation}
Then, we use  \eqref{D-12-0} and \eqref{D-12-1} to get 
\begin{align*}
\begin{aligned}
&\limsup_{t \to \infty} d(I_n,R_i(t)) \\
& \hspace{1cm} \le 2\arcsin\Big[ \frac12 \arcsin\Big(\frac{N\|\Omega_i\|}{\kappa \tilde{I}(\gamma)}\Big)\Big ] < 2\arcsin\Big[ \frac12 \arcsin \Big(
\sin \Big (2\sin\Big (\frac{\gamma}{2}\Big) \Big )\Big) \Big ]  \\
& \hspace{1cm} = 2\arcsin \Big( \sin \Big( \frac{\gamma}{2} \Big) \Big) = \gamma.
\end{aligned}
\end{align*}
Note that $\arcsin \circ\sin$ can be canceled since by framework $(\mathcal{F}_A1)$, 
\[
\gamma = 2\arcsin \Big (\frac{\gamma_0}{2} \Big ) \implies
2\sin\Big (\frac{\gamma}{2}\Big)=\gamma_0<2\sin\Big(\frac{\beta}{2}\Big)\le \beta<\frac{\pi}{2}.
\]
\end{proof}
In the next corollary, we can relax the condition \eqref{D-10-4} for the coupling strength $\kappa$ and the initial data, if there are $N_0$ particles lying in the open ball $B_{\gamma_0}$ initially.
\begin{corollary}\label{C3.2}
Suppose that initial data and system parameters satisfy $({\mathcal F}_A1)$ and additional conditions:~there exists an index set $K$ with a cardinality $N_0$ such that $R_l^0\in B_{\gamma_0}$ for every $l\in K$, and
\begin{equation} \label{D-12-2}
\begin{aligned}
&\kappa >  \kappa_c(\gamma, N_0) :=  \frac{N}{N_0} \frac{\max_{j \in [N]} \|\Omega_j\|}{\sin(2\sin \frac{\gamma}{2}) \tilde I(\gamma)}, \quad R_i^0 \in B_{\Gamma_i(\gamma)}, \quad \forall~i \in [N], 
\end{aligned}
\end{equation}
and let $\{R_i\}$ be a global solution to \eqref{A-2}. Then we have
\[
    \limsup_{t \to \infty} d(I_n,R_j(t)) < 2\arcsin\Big(\frac12\arcsin\Big(\frac{\kappa_c(\gamma, N_0)}{\kappa}\sin(2\sin \frac{\gamma}{2})\Big)\Big) < \gamma, \quad \forall\ j\in [N].
\]
\end{corollary}
\begin{proof} We use the same argument in Step A in the proof of Proposition \ref{P3.2} together with the relation
\[ \sum_{j=1}^N \tilde{I}(d(I_n, R_j)) \geq \sum_{l=1}^{N_0} \tilde{I}(d(I_n, R_l))  \geq N_0 \tilde{I}(\gamma), \]
to find 
\[
\frac{d}{dt} (n-\tr(R_i)) \le 2\sqrt{\sum_{j=1}^{m_l} \sin^2 \theta_{ij}} \Big(\|\Omega_i\| - \frac{\kappa N_0}{N} \tilde{I}(\gamma) \sin d(I_n, R_i)\Big).
\] 
Then, we continue the same argument in the proof of Corollary \ref{C3.1} and we obtain the desired result.
\end{proof}
\begin{remark} For $N_0 = 1$, the conditions in \eqref{D-12-2} reduce to \eqref{D-10-4}.
\end{remark}
\subsection{$\ell^1$-exponential stability} \label{sec:3.2}
In this subsection, we study the uniform(-in-time) $\ell^1$-stability of solutions which are contained in $B_\gamma$. For ${\mathbf R} = (R_1, \cdots, R_N)\in (\bbr^{n\times n})^N$, we simply set 
\[ \| {\mathbf R} \|_{\ell^1} = \sum_{i=1}^{N} \| R_i \|.  \]
First, we provide several lemmas to be used in the proof of exponential $\ell^1$-stability. \begin{lemma} \label{L3.4}
Let $R$ be an $n \times n$ real matrix. Then, $R^TR$ is positive semi-definite.
\end{lemma}
\begin{proof} Let $x \in {\mathbb R}^n$ be any. Then, we have
\[ x^T R^T R x = (Rx)^T (Rx) \geq 0. \]
Therefore, $R^T R$ is positive semi-definite. 
\end{proof}
\begin{lemma} \label{L3.5}
Let $R$ be an $n \times n$ real matrix in $SO(n)$. Then, the eigenvalues of $\frac{R + R^T}{2}$ are $1$ and $\cos \theta_{j}$, for $j \in \{1, 2, ... ,m\}$ with an integer $m\in [0,n/2]$.
Moreover, for $R\ne I_n$, $\frac{R+R^T}{2} \succeq \Big( \min_{j\in [m]} \cos \theta_{j} \Big) \, I_n$.
\end{lemma}

\begin{proof}
For $R\ne I_n$, we use (1) of Lemma~\ref{L2.3} to obtain
        \[
        R \sim  \begin{pmatrix}
            {\mathcal R}(\theta_1)   \\
             &  \ddots   \\
              &  &   {\mathcal R}(\theta_m) \\
              &  & & I_{n-2m}
        \end{pmatrix},\qquad
        {\mathcal R}(\theta):= \begin{pmatrix}
        \cos\theta & -\sin\theta \\
        \sin\theta & \cos\theta
    \end{pmatrix}.
        \]
    So 
    \[
    \frac{R+R^T}{2}\sim \operatorname{diag}(\cos\theta_1 I_2, \dots \cos\theta_{m}I_2, I_{n-2m}),
    \]
    which shows both claims. 
\end{proof}

\begin{theorem}\label{T3.1}
Suppose system parameters satisfy the following conditions:
\[ 0 < \beta < \frac{\pi}{2}, \quad 0 < \gamma <  \beta, \quad \lambda_1:=\kappa(\cos \gamma\,\tilde{I}(\gamma)-\gamma\,\text{Lip*}I)>0,
\]
and let $\mathbf R=\{R_i\}_{i=1}^N$ and $\mathbf S=\{S_i\}_{i=1}^N$ be global solutions to \eqref{A-2} with initial data $\{R_i^0\}_{i=1}^N$ and $\{S_i^0\}_{i=1}^N$, respectively satisfying a priori conditions:
\begin{equation} \label{D-13-1}
\{R_i(t)\}_{i=1}^N,\ \{S_i(t)\}_{i=1}^N \subset \overline{B_{\gamma}},\quad t\ge 0.
\end{equation}
Then, we have exponential $\ell^1$-stability:
\[
\|\mathbf R(t)-\mathbf S(t)\|_{\ell^1}\le e^{-\lambda_1t}\,\|\mathbf R^0-\mathbf S^0\|_{\ell^1},
\quad t\ge 0.
\]
\end{theorem}
\begin{proof}

We set 
\[
\langle I^R(t) \rangle:=\frac1N\sum_{j=1}^N I(R_j(t)) \quad \mbox{and} \quad \langle I^S(t) \rangle:=\frac1N\sum_{j=1}^N I(S_j(t)).
\]
Then, it follows from the decreasing property of $I = I(\cdot)$ and the a priori conditions \eqref{D-13-1} that 
\begin{equation}\notag \label{D-13-2}
\langle I^R(t) \rangle \geq I(\gamma) \quad \mbox{and} \quad  \langle I^S(t) \rangle \geq I(\gamma).
\end{equation}
Recall that $R_i$ and $S_i$ satisfy 
\begin{equation} \label{D-13-3}
\frac{dR_i}{dt} = \Omega_i R_i + \frac{\kappa}{2} \langle I^R \rangle (I_n - R_i^2), \quad \frac{dS_i}{dt} = \Omega_i S_i + \frac{\kappa}{2} \langle I^S \rangle (I_n - S_i^2).
\end{equation}
Then, the relation $\eqref{D-13-3}_1 - \eqref{D-13-3}_2$ implies 
\begin{align*}
\frac{d}{dt}(R_i-S_i)
&= \Omega_i(R_i-S_i)
+\frac{\kappa}{2}\Big( \langle I^R \rangle (I_n-R_i^{2})- \langle I^S \rangle (I_n-S_i^{2})\Big) \\
&= \Omega_i(R_i-S_i) -\frac{\kappa}{2} \langle I^R \rangle (R_i^{2}-S_i^{2}) +\frac{\kappa}{2} \Big (\langle I^R \rangle -\langle I^S \rangle \Big)(I_n-S_i^{2}).
\end{align*}
Hence, we have
\begin{equation}\notag\label{D-14}
\begin{aligned}
 \frac12\frac{d}{dt}\|R_i-S_i\|^{2} &=\Big\langle R_i-S_i,\frac{d}{dt}(R_i-S_i)\Big\rangle \\
&= \Big \langle R_i-S_i,\Omega_i(R_i-S_i) \Big \rangle
-\frac{\kappa}{2}  \langle I^R \rangle  \Big \langle R_i-S_i, R_i^{2}-S_i^{2} \Big \rangle \\
&\hspace{0.2cm} +\frac{\kappa}{2} ( \langle I^R \rangle -\langle I^S  \rangle ) \Big \langle R_i-S_i,\,I_n-S_i^{2} \Big \rangle \\
&=: {\mathcal I}_{21} + {\mathcal I}_{22} + {\mathcal I}_{23}.
\end{aligned}
\end{equation}
Below, we estimate the terms ${\mathcal I}_{2i},~i = 1,2,3$ one by one. \newline

\noindent $\bullet$~(Estimate of ${\mathcal I}_{21}$):  We use the skew-symmetry of $\Omega_i$ to get 
\[ \langle M,\Omega_i M\rangle =0 \quad \mbox{for all $M\in\mathbb{R}^{n\times n}$}. \]
This yields
\begin{equation} \label{D-14-1}
{\mathcal I}_{21} = 0.
\end{equation}
\vspace{0.2cm}
\noindent $\bullet$~(Estimate of ${\mathcal I}_{22}$):~We use  
\begin{equation}\notag \label{D-14-2}
R_i^2-S_i^2=(R_i-S_i)(R_i+S_i)-[R_i,S_i].
\end{equation}
to see
\begin{align}
\begin{aligned} \label{D-14-2-1}
- \Big \langle R_i-S_i, R_i^{2}-S_i^{2} \Big \rangle &= -\Big \langle R_i-S_i, (R_i-S_i)(R_i+S_i) \Big \rangle + \Big \langle R_i-S_i, [R_i,S_i] \Big \rangle \\
 &={\mathcal I}_{221} + {\mathcal I}_{222}.
\end{aligned}
\end{align}
\noindent $\diamond$ (Estimate of  ${\mathcal I}_{222}$):~We use 
\begin{equation}\label{D-14-2-2} R_i^T R_i = I_n, \quad  S_i^T S_i = I_n, \quad \tr(A \pm B) = \tr(A) \pm \tr(B), \quad  \tr (AB) = \tr(BA) \end{equation}
to find 
\begin{align}
\begin{aligned} \label{D-14-3}
{\mathcal I}_{222} &= \langle R_i-S_i,\,[R_i,S_i]\rangle \\
&= \langle R_i-S_i,\,R_iS_i-S_iR_i\rangle = \frac12\tr\!\big((R_i^T-S_i^T)(R_iS_i-S_iR_i)\big) \\
&= \frac12\tr\!\big(S_i - R_i^T S_i R_i - S_i^T R_i S_i + R_i\big)  \\
&= \frac12\big(\tr(S_i)-\tr(S_i)-\tr(R_i)+\tr(R_i)\big)=0.
\end{aligned}
\end{align}
\noindent $\diamond$ (Estimate of  ${\mathcal I}_{221}$):~We use $\eqref{D-14-2-2}$ to obtain 
\begin{align}
\begin{aligned} \label{D-14-4}
 {\mathcal I}_{221} &=-\langle R_i-S_i,\,(R_i-S_i)(R_i+S_i)\rangle \\
&=-\frac12\tr\!\big((R_i^T-S_i^T)(R_i-S_i)(R_i+S_i)\big) \\
&=-\frac12\tr\!\Big((R_i^T-S_i^T)(R_i-S_i)\frac{R_i+S_i+R_i^T+S_i^T}{2}\Big) \\
&~~-\frac12 \tr\!\Big((R_i^T-S_i^T)(R_i-S_i)\frac{R_i+S_i - R_i^T - S_i^T}{2}\Big) \\
&=  {\mathcal I}_{2211}  +  {\mathcal I}_{2212},
\end{aligned}
\end{align}
where we use decomposition $R_i + S_i$ as the sum of symmetric part and skew-symmetric part:
\[
R_i+S_i = \frac{R_i + S_i + R_i^T + S_i^T}{2} +  \frac{R_i + S_i - R_i^T - S_i^T}{2}.
\]
Next, we estimate the terms ${\mathcal I}_{221i},~i=1,2$ one by one. \newline

\noindent $\clubsuit$ (Estimate of ${\mathcal I}_{2211}$):~By Lemma \ref{L3.5}, for $R_i,S_i\ne I_n$, we have
\begin{equation} \label{D-14-4-1}
\frac{R_i+R_i^T}{2} \succeq  \Big( \min_{j\in [m_i] } \cos \theta_{ij} \Big) I_n, \quad \frac{S_i+S_i^T}{2} \succeq  \Big( \min_{j\in [ {\tilde m}_i]} \cos {\tilde \theta}_{ij} \Big) I_n.       
\end{equation}
On the other hand, it follows from \eqref{B-5-1} and \eqref{D-13-1} that 
\[     d(I_n, R_i) = \sqrt{\sum_{j=1}^{m_i} \theta_{ij}^2} \leq \gamma, \quad d(I_n, S_i) = \sqrt{\sum_{j=1}^{{\tilde m}_i} {\tilde \theta}_{ij}^2} \leq \gamma. \]
This yields for any $R_i,S_i$,
\begin{equation} \label{D-14-4-2}
 \min_{j\in [m_i] } \cos \theta_{ij}  \geq \cos \gamma \quad \mbox{and} \quad \min_{j\in [ {\tilde m}_i]}  \cos {\tilde \theta}_{ij} \geq \cos \gamma. 
\end{equation}
Now, we use \eqref{D-14-4-1} and \eqref{D-14-4-2} to get 
\begin{equation} \label{D-14-4-3}
\frac{R_i+R_i^T}{2} \succeq  \cos \gamma I_n, \quad \frac{S_i+S_i^T}{2} \succeq \cos \gamma I_n.
\end{equation}
Note that by Lemma \ref{L3.4}, 
\begin{equation} \label{D-14-4-4}
\mbox{$(R_i^T-S_i^T)(R_i-S_i)$ is symmetric and positive semi-definite.}
\end{equation}
We use \eqref{D-14-4-3}, \eqref{D-14-4-4} and  the sixth assertion in Lemma \ref{L2.3} to see
\begin{align*}
\begin{aligned} 
& \frac{1}{2} \tr\!\Big((R_i^T-S_i^T)(R_i-S_i)\frac{R_i+ R_i^T }{2}\Big) \geq \frac{\cos \gamma}{2}  \tr\!\Big((R_i^T-S_i^T) (R_i - S_i) \Big) =  \cos \gamma  \| R_i - S_i \|^2, \\
& \frac{1}{2} \tr\!\Big((R_i^T-S_i^T)(R_i-S_i)\frac{S_i + S_i^T}{2}\Big) \geq  \frac{\cos \gamma}{2}  \tr\!\Big((R_i^T-S_i^T)(R_i-S_i) \Big) =  \cos \gamma  \| R_i - S_i \|^2.
\end{aligned}
\end{align*}
Therefore, we have
\begin{equation} \label{D-14-4-5}
{\mathcal I}_{2211} \leq -2 \cos \gamma  \| R_i - S_i \|^2.  
\end{equation}

\vspace{0.2cm}

\noindent $\clubsuit$  (Estimate of ${\mathcal I}_{2212}$):  Since the matrix $(R_i^T-S_i^T)(R_i-S_i)\cdot \frac{R_i+S_i - R_i^T - S_i^T}{2}$ is is a multiplication of symmetric matrix and skew symmetric matrix, its trace is zero:
\begin{equation} \label{D-14-4-6}
{\mathcal I}_{2212} = \operatorname{tr}(AB)=-\operatorname{tr}(BA)=-\operatorname{tr}(AB)= 0. 
\end{equation}
In \eqref{D-14-4}, we combine \eqref{D-14-4-5} and \eqref{D-14-4-6} to find 
\begin{equation} \label{D-14-4-7}
 {\mathcal I}_{221} \leq -2 \cos \gamma  \| R_i - S_i \|^2.
\end{equation}
Therefore, we use \eqref{D-14-2-1}, \eqref{D-14-3} and \eqref{D-14-4-7}  to find 
\begin{align}\label{D-15}
\begin{aligned}
&  {\mathcal I}_{22} = - \frac{\kappa}{2} \langle I^R \rangle \tr\!\Big((R_i^T-S_i^T)(R_i-S_i)\frac{R_i+S_i+R_i^T+S_i^T}{2}\Big) \\
& \hspace{1cm} \le -\frac{\kappa}{2}   \tilde{I}(\gamma) \tr\!\big((R_i^T-S_i^T)(R_i-S_i)\,2\cos \gamma\, I_n\big) \\
&\hspace{1cm} = -\kappa \tilde{I}(\gamma)\cos \gamma\,\|R_i-S_i\|^2 .
\end{aligned}
\end{align}

\noindent $\bullet$~(Estimate of ${\mathcal I}_{23}$):~We use the Cauchy-Schwarz inequality to get 
\begin{align}
\begin{aligned} \label{D-15-1}
|{\mathcal I}_{23} | &\leq  \frac{\kappa}{2} | \langle I^R \rangle -\langle I^S \rangle | \times |\langle R_i-S_i,\,I_n-S_i^{2}\rangle|  \\
&\leq  \frac{\kappa}{2N}\|R_i - S_i\| \|I_n - S_i^2\| \sum_{j=1}^N |I(R_j) - I(S_j)|,
\end{aligned}
\end{align}
where we used the relation:
\[|\langle I^R \rangle -\langle I^S \rangle | = \Big| \frac1N\sum_{j=1}^N I(R_j) - \frac1N\sum_{j=1}^N I(S_j) \Big| \leq \frac1N\sum_{j=1}^N | I(R_j) -  I(S_j)|.\]
Since $I_n-S_i^2$ is normal, its singular values are equal to the absolute value of its eigenvalues.
Thus, we have
\begin{align}
\begin{aligned} \label{D-15-1-1}
\|I_n-S_i^2\|&=\sqrt{\frac{1}{2}\tr\big((I_n-S_i^2)^T(I_n-S_i^2)\big)} =\sqrt{\frac{1}{2}\sum_{j=1}^{m_i}\Big(|1-e^{2\mathrm{i}\theta_{ij}}|^2+|1-e^{-2\mathrm{i}\theta_{ij}}|^2 \Big)}\\
&=\sqrt{\frac12 \sum_{j=1}^{m_i}\Big(|e^{-\mathrm{i}\theta_{ij}}-e^{\mathrm{i}\theta_{ij}}|^2+|e^{\mathrm{i}\theta_{ij}}-e^{-\mathrm{i}\theta_{ij}}|^2 \Big)} =2\sqrt{\sum_{j=1}^{m_i}\sin^2\theta_{ij}} \\
&\le 2d(I_n, S_i)\le 2\gamma.
\end{aligned}
\end{align}
On the other hand, we have
\begin{equation} \label{D-15-1-2}
|I(R_j) - I(S_j)| \leq \text{Lip*}(I) \|R_j - S_j \|.
\end{equation}
We plug \eqref{D-15-1-1}, \eqref{D-15-1-2} into \eqref{D-15-1} to obtain

\begin{equation}\label{D-15-2}
	|\mathcal{I}_{23}| \le \frac{\kappa\gamma}N \text{Lip}^*I \|R_i-S_i\|\sum_{j=1}^{N}\|R_j - S_j\|.
\end{equation}

We use \eqref{D-14-1}, \eqref{D-15} and \eqref{D-15-2} to find 
\begin{equation} \label{D-15-3}
\frac{d}{dt}\|R_i-S_i\|
\le -\kappa \tilde{I}(\gamma)\cos \gamma\,\|R_i-S_i\|
+\frac{\kappa \gamma}{N}\text{Lip*}(I)\sum_{j=1}^N \|R_j-S_j\|.
\end{equation}
We sum up \eqref{D-15-3} over all $i \in [N]$ to obtain
\[
\frac{d}{dt}\|\mathbf R-\mathbf S\|_{\ell^1}
\le -\kappa \Big(\tilde{I}(\gamma)\cos \gamma-\gamma\text{Lip*}I \Big) \|\mathbf R-\mathbf S\|_{\ell^1} =: -\lambda_1 \|\mathbf R-\mathbf S\|_{\ell^1}.
\]
This yields the exponential stability:
\[
\|\mathbf R(t)-\mathbf S(t)\|_{\ell^1}\le e^{-\lambda_1t}\|\mathbf R^0-\mathbf S^0\|_{\ell^1},
\quad t\ge 0.
\]
\end{proof}
\begin{remark}
(Exponential $\ell^1$-stability $\quad \Longrightarrow \quad$ Existence of a unique equilibrium in $B_\gamma$)
In the sequel, we briefly show that there exists a unique equilibrium of \eqref{A-2} using the uniform-stability estimate and exponential relaxation toward a unique equilibrium. In particular, by Proposition~\ref{P3.1}, there exists a solution contained in $B_\gamma$; call it $\mathbf{R}(t)$. Next, we show that it converges to some equilibrium point.
Since $\mathbf{R}(t+s)$ are also solutions contained in $B_\gamma$,  by Proposition~\ref{P3.2}, 
\[
\| {\mathbf R}(n+1)- {\mathbf R}(n)\|_{\ell^1} \le e^{-\lambda_1 n}\| {\mathbf R}(1)- {\mathbf R}(0)\|_{\ell^1}.
\]
So the sequence $\{R(n)\}_{n=1}^\infty$ is a Cauchy sequence; by completeness of $SO(n)$, it converges to some point $R^*\in \overline{B_\gamma}$:
\[ \lim_{n \to \infty} \| {\mathbf R}(n) - R^* \| = 0. \]
Using the following fact, we can show that $R(t)$ itself actually converges to $R^*$:~For $t \geq 0$, 
\[
\sup_{s\in [0,1]}\| {\mathbf R}(t+s)- {\mathbf R}(t)\|_{\ell^1} \le e^{-\lambda_1 t}\sup_{s\in [0,1]}\| {\mathbf R}(s)- {\mathbf R}(0)\|_{\ell^1}
< \infty.
\] 
By the continuity of the dynamical system, $R^*$ must be the equilibrium point, and any solution contained in $B_\gamma$ must converge to this point, because of exponential $\ell^1$-stability. This also gives the uniqueness of equilibrium point, too.
\end{remark}

\section{Existence of equilibria} \label{sec:4}
\setcounter{equation}{0}
In this section, we present an alternative proof of the existence of an equilibrium in the ball $B_\gamma$. The uniqueness of that equilibrium and the exponential relaxation toward it will then follow from the uniform-stability result obtained in Section~\ref{sec:3}. First, we provide several useful lemmas which will be used in the proof.
\begin{lemma}\label{L4.1}
	Suppose that
	\[\mathrm{diag}(a_1J, \cdots a_p J, O_{n-2p})\sim \mathrm{diag}(b_1J, \cdots, b_qJ, O_{n-2q}), \quad J = \begin{pmatrix}
 0 & -1 \\ 1 & 0	
 \end{pmatrix}
\]
for some nonzero constants $a_1, \cdots, a_p$, $b_1, \cdots, b_q$. Then we have \[p=q \quad \mathrm{and}\quad \{a_1^2, \cdots ,a_p^2\} = \{b_1^2, \cdots ,b_q^2 \}\]
as a multisets. Equivalently, the two collections have the same elements with the same multiplicities.
\end{lemma}
\begin{proof}
We postpone its proof to Appendix \ref{App-A-2}.	
\end{proof}
\noindent Suppose that $\{R^{\infty}_i\}_{i=1}^N$ is an equilibrium of the system~\eqref{A-2}:
\begin{equation}\label{D-16}
\Omega_i R^{\infty}_i + \frac{\kappa \langle I^{\infty} \rangle}{2} \Big (I_n - (R^{\infty}_i)^2 \Big ) = 0, \quad \forall~i \in [N],  \quad \mbox{where} \displaystyle \langle {I}^{\infty} \rangle  = \frac{1}{N} \sum_{i=1}^N  I(R^{\infty}_i).
\end{equation}
Now, we use $R^{\infty}_i (R^{\infty}_i)^T = I_n$ and dividing both side with $\kappa \langle I^\infty\rangle$ to rewrite \eqref{D-16} into
\begin{equation} \label{D-16-1}
\frac{\Omega_i}{\kappa \langle I^\infty \rangle} = \frac{R^{\infty}_i - (R^{\infty}_i)^T}{2}, \quad \forall~i \in [N].
\end{equation}

\noindent From (2) in Lemma~$\ref{L2.3}$,  each $R^{\infty}_i $ is orthogonally similar to $\Lambda^{\infty}_i$:~there exists $P_i \in O(n)$ such that 
\begin{equation} \label{D-16-2}
\Lambda^{\infty}_i = \begin{pmatrix}
        \cos \theta^{\infty}_{i1} & -\sin \theta^{\infty}_{i1}\\
         \sin \theta^{\infty}_{i1} & \cos \theta^{\infty}_{i1} \\
         & & \ddots \\
         & & & & \cos \theta^{\infty}_{im_i} & -\sin \theta^{\infty}_{im_i} \\
        &  & & & \sin \theta^{\infty}_{im_i} & \cos \theta^{\infty}_{im_i} \\
        & & & & & & I_{n-2m_i}
    \end{pmatrix}, \quad  R^{\infty}_i = P_i^T \Lambda^{\infty}_i P_i.  
 \end{equation} 
This shows that
\[\frac{R_i^\infty - (R_i^\infty)^T}{2} \sim \mathrm{diag}(\sin \theta_{i1}^\infty J, \cdots, \sin \theta_{in_i}^\infty J, O_{n-2n_i}), \quad J = \begin{pmatrix}
    0 & -1 \\ 1 & 0
\end{pmatrix}.\] 
Note that we are seeking $R_i^{\infty}$ with $R_i^\infty\in B_\gamma$ with $\gamma<\frac\pi 2$, we may assume that $0 < \theta_{ij} < \frac \pi 2$, hence $\sin\theta_{ij}^\infty\ne 0$ for every $j\in [m_i]$.

On the other hand, since $\Omega_i$ is skew-symmetric, we use (4) in Lemma~$\ref{L2.3}$, to say that there exist positive constants $\lambda_{i1}, \cdots, \lambda_{in_i}$ such that
\begin{equation}\notag\label{D-17}
    \Omega_i\sim \text{diag}(\lambda_{i1} J, \cdots, \lambda_{in_i}J, O_{n-2n_i}),
\end{equation}
which yields
\begin{equation}\notag
	\frac{\Omega_i}{\kappa \langle I^\infty \rangle}\sim \text{diag}\Big(\frac{\lambda_{i1}}{\kappa \langle I^\infty \rangle} J, \cdots, \frac{\lambda_{in_i}}{\kappa \langle I^\infty \rangle}J, O_{n-2n_i}\Big)
\end{equation}
Then by Lemma \ref{L4.1}, $n_i = m_i$ and 
\begin{equation}\label{D-17-1}
\{\sin^2 \theta_{i1}, \cdots, \sin^2 \theta_{im_i}\} = \Big\{\frac{\lambda_{i1}^2}{\kappa^2 \langle I^\infty \rangle^2}, \cdots, \frac{\lambda_{in_i}^2}{\kappa^2 \langle I^\infty \rangle^2}\Big\}	
\end{equation}

as a multisets. Now from the $\theta_{ij}<\frac\pi 2$ assumption, we have
\[\{\theta_{i1}, \cdots \theta_{im_i}\} = \Big\{\arcsin\Big(\frac{\lambda_{i1}}{\kappa \langle I^\infty \rangle}\Big), \cdots, \arcsin \Big(\frac{\lambda_{in_i}}{\kappa \langle I^\infty \rangle}\Big)\Big\}.\]
as a multisets. This gives the important relation:
\begin{equation}\notag \label{D-18-1}
        d(I_n, R^{\infty}_i) = \sqrt{\sum_{j=1}^{m_i}(\theta^{\infty}_{ij})^2} = \sqrt{\sum_{j=1}^{m_i}\arcsin^2\Big(\frac{\lambda_{ij}}{\kappa \langle { I}^{\infty} \rangle}\Big)}.
\end{equation}
This yields the equation for $\langle{  I}^{\infty} \rangle$:
\begin{equation}\label{D-19}
  \langle { I}^{\infty} \rangle = \frac 1 N \sum_{i=1}^N \tilde I\left(\sqrt{\sum_{j=1}^{m_i}\arcsin^2\Big(\frac{\lambda_{ij}}{\kappa \langle { I}^{\infty} \rangle}\Big)}\right).
\end{equation} 
Motivated by \eqref{D-19}, we set 
\begin{equation}\notag \label{D-19-0}
f(x) \coloneqq  \frac 1 N \sum_{i=1}^N \tilde I\left(\sqrt{\sum_{j=1}^{m_i}\arcsin^2\Big(\frac{\lambda_{ij}}{\kappa x}\Big)}\right).
\end{equation}
Then $\langle { I}^{\infty} \rangle$ is the fixed point of $f$:
\begin{equation} \label{D-19-1}
\langle {I}^{\infty} \rangle =f( \langle { I}^{\infty} \rangle).
\end{equation}
To verify the solvability of \eqref{D-19-1} and existence of equilibrium ${\mathbf R}^{\infty} = \{ {R}^{\infty}_i \}$ in $B_\gamma$, we proceed in following three steps: \newline
\begin{itemize}
    \item 
   Step A:~Find $g(x)$ and $h(x)$ such that
    \begin{equation}\notag \label{D-19-2}
    g(x)\le f(x)\le h(x).
    \end{equation}
    \item
   Step B:~Show the existence of $x_0, x_1$ such that
    \begin{equation}\notag \label{D-19-3}
    x_0 < x_1, \quad x_0 \le g(x_0),\quad x_1 \ge h(x_1).
     \end{equation}
    These yield
    \[ 
    x_0 \leq f(x_0), \quad x_1 \geq f(x_1).
    \]
    Then, the continuous map $x \mapsto x - f(x)$ changes sign on $[x_0, x_1]$,
    so by the intermediate value theorem, there exists $x^*\in[x_0, x_1]$ such that
    \[x^* = f(x^*).\]
    This shows the existence of $\langle {I}^{\infty} \rangle$ satisfying \eqref{D-19-1}.
    Furthermore, impose a condition so that this gives the equilibrium in $B_\gamma$.
    \vspace{0.1cm}
    \item
     Step C:~Finally, one can determine $R^{\infty}_i$ from $\langle I^\infty \rangle$. We construct the solution  $R_i^\infty$ to \eqref{D-16-1} under the appropriate condition on $\kappa$ and $\langle I^\infty\rangle$.
\end{itemize}

\vspace{0.2cm}

In the next three subsections, we show the above steps one by one. 
\subsection{Construction of $g$ and $h$}  
In this subsection, we perform Step A described above. By $\eqref{D-16-1}$ and the block representations of $R_i^\infty$ and $\Omega_i$, we have
\[
\frac{\Omega_i}{\kappa\langle I^\infty\rangle}
\sim
\operatorname{diag}(\sin\theta_{i1}^\infty J,\dots,\sin\theta_{im_i}^\infty J,O_{n-2m_i}),
\]
and therefore\begin{equation}\label{D-19-4}
\frac{\|\Omega_i\|}{\kappa\langle I^\infty\rangle}
=
\sqrt{\sum_{j=1}^{m_i}\sin^2\theta_{ij}^\infty}.
\end{equation} 
Since $0\le \theta_{ij}^\infty\le d(I_n,R_i^\infty)<\gamma<\frac{\pi}{2}$, we have
\[
\frac{\sin d(I_n,R_i^\infty)}{d(I_n,R_i^\infty)}\,\theta_{ij}^\infty
\le
\sin\theta_{ij}^\infty
\quad \forall\, j\in[m_i].
\]
We square both side and sum over $j$ to obtain
\[
\sin d(I_n,R_i^\infty)
\le
\sqrt{\sum_{j=1}^{m_i}\sin^2\theta_{ij}^\infty}
\le
\sqrt{\sum_{j=1}^{m_i}(\theta_{ij}^\infty)^2}
=
d(I_n,R_i^\infty).
\]
Combining this with \eqref{D-19-4}, we obtain
\begin{equation}\label{D-20}
\frac{\|\Omega_i\|}{\kappa\langle I^\infty\rangle}
\le
d(I_n,R_i^\infty)
\le
\arcsin\Bigl(\frac{\|\Omega_i\|}{\kappa\langle I^\infty\rangle}\Bigr).
\end{equation}
We use the decreasing property of $\tilde{I}$ to find 
\begin{equation*}
\frac{1}{N}\sum_{j=1}^N \tilde{I}\left(\arcsin\frac{\|\Omega_j\|}{\kappa \langle {I}^{\infty} \rangle}\right) 
\le \frac 1 N \sum_{j=1}^{N} \tilde I (d(I_n, R^{\infty}_j))
\le \frac{1}{N}\sum_{j=1}^N \tilde{I}\left(\frac{\|\Omega_j\|}{\kappa \langle I^{\infty} \rangle}\right).
\end{equation*}
Now, we set 
\[g(x) := \frac{1}{N}\sum_{j=1}^N \tilde{I}\left(\arcsin\frac{\|\Omega_j\|}{\kappa x}\right), \quad h(x) := \frac{1}{N}\sum_{j=1}^N \tilde{I}\left(\frac{\|\Omega_j\|}{\kappa x}\right)
\]
to see
\[ g(x)\le f(x)\le h(x). \]

\medskip

\subsection{Existence of $x_0$ and $x_1$}
Since we seek an equilibrium in $(B_\gamma)^N$, we require $d(I_n,R_i^\infty)<\gamma$ for every $i\in[N]$. 
By \eqref{D-20}, this condition is satisfied if
\[
\arcsin\left(\frac{\|\Omega_i\|}{\kappa\langle I^\infty\rangle}\right)<\gamma,
\]
which is equivalent to
\[
\langle I^\infty\rangle
>
\frac{\|\Omega_i\|}{\kappa\sin\gamma},
\qquad i\in[N].
\]
We set
\[ G(x) = x-g(x), \quad H(x) = x-h(x). \]
Then, it follows from the boundedness of $\tilde I$ that 
\[\lim_{x\rightarrow\infty}H(x) = \infty.\]
Hence, $x_1$ can be obtained by selecting sufficiently big number; so it is enough to show the existence of $x_0$ such that
\[x_0 > \frac{\max_{j\in[N]}\|\Omega_j\|}{\kappa\sin\gamma}\quad \text{and}\quad G(x_0)\le 0.\]
Suppose that the coupling strength $\kappa$ satisfies
\[
\kappa  > \frac{\max_{j\in [N]}\|\Omega_j\|}{\sin\gamma\, \tilde I(\gamma)}.
\]
Then, since $\tilde I$ is a decreasing function, we obtain
\begin{align*}
G\Big(\frac{\max_{j\in [N]}\|\Omega_j\|}{\kappa\sin\gamma}\Big) &= \frac{\max_{j\in [N]}\|\Omega_j\|}{\kappa\sin\gamma} - \frac{1}{N}\sum_{j=1}^N \tilde{I}\left(\arcsin\left(\frac{\|\Omega_j\|}{\max_{j\in [N]}\|\Omega_j\|}\sin\gamma\right)\right)\\
&< \tilde I(\gamma)-\frac 1 N \sum_{j=1}^N\tilde I(\gamma)=0.
\end{align*}
Therefore, by continuity of $G$ we may take 
\[
x_0:=\frac{\max_{j\in[N]}\|\Omega_j\|}{\kappa\sin\gamma}+\varepsilon,
\]
for small enough $\varepsilon>0$ such that 
\[x_0 > \frac{\max_{j\in[N]}\|\Omega_j\|}{\kappa\sin\gamma}\quad \text{and}\quad G(x_0)\le 0.\]
Consequently,
\[
x_0-f(x_0)\le x_0-g(x_0)\le 0,
\qquad
x_1-f(x_1)\ge x_1-h(x_1)\ge 0.
\]
By intermediate value theorem, 
\[\exists\ x^*\in[x_0, x_1] \quad \text{such that}\quad x^* = f(x^*).\]
Thus $\langle I^\infty\rangle=x^*$ solves \eqref{D-19-1}. 

\subsection{Finding $R_i^\infty$ from $\langle I^\infty \rangle$}
Now we have the existence of $\langle I^\infty\rangle$ that solves \eqref{D-19-1} and satisfies
\begin{equation}\notag\label{D-21}
\langle I^\infty \rangle > \frac{\max_{j\in[N]}\|\Omega_j\|}{\kappa \sin\gamma}.
\end{equation}
 
The equilibrium $R_i^\infty$ with given $\langle I^\infty \rangle$ can be explicitly written as follows:
for a orthogonal matrix $P_i$ that
	\[\Omega_i \sim P_i^T\mathrm{diag}(\lambda_{i1}J, \cdots, \lambda_{im_i}J, O_{n-2m_i}) P_i,\]
We set
\[R_i^\infty =  P_i^T \Lambda_i^\infty P_i\]
where $\Lambda_i$ is defined in \eqref{D-16-2}, with
\[\theta_{ij}^\infty = \arcsin \Big(\frac{\lambda_{ij}}{\kappa \langle I^\infty \rangle}\Big).\]
Then we have
\begin{equation*}
\begin{aligned}
    \frac{R_i^\infty - (R_i^\infty)^T}{2} 
    &= P_i^T \mathrm{diag}(\sin\theta_{i1}J, \cdots, \sin\theta_{im_i}J, O_{n-2m_i})P_i^T \\
    &= P_i^T \mathrm{diag}\Big(\frac{\lambda_{i1}}{\kappa \langle I^\infty \rangle}, \cdots, \frac{\lambda_{im_i}}{\kappa \langle I^\infty \rangle}, O_{n-2m_i}\Big) P_i =\frac{\Omega_i}{\kappa \langle I^\infty \rangle}
\end{aligned}
\end{equation*}
Hence $R_i^\infty$ is a solution to $\eqref{D-16-1}$.

\begin{remark}
We don't have the uniqueness of the solution $\langle I^\infty \rangle$ to \eqref{D-19}.
Hence the uniqueness of the equilibrium is not ensured. 
Although the uniqueness in the subset $B_\gamma$ will be shown in the following theorem, we can actually show the possibility of uncountably many equilibrium in homogeneous Winfree model on full $SO(n)$ in Section~\ref{sec:5.2}.
\end{remark}

Next, we are ready to show the uniqueness of the equilibrium in $B_\gamma$, and provide an estimate on relaxation toward the equilibrium.
\begin{theorem}\label{T4.1}
Suppose that the initial data and parameters satisfy $({\mathcal F}_A1)$ and, in addition, there exists an index $i_* \in [N]$ such that $ R_{i_*}^0 \in B_{\gamma_0},$ and
\begin{equation}\notag\label{D-16-3}
\begin{aligned}
&R_i^0 \in B_{\Gamma_i(\gamma)}, \quad \forall\, i \in [N], \quad \lambda_1:= \kappa(\cos \gamma\,\tilde I(\gamma)-\gamma\,\text{Lip*}(I)) > 0, \\
&\kappa > \kappa_c(\gamma,1) := \frac{N\max_{j\in[N]}\|\Omega_j\|}{\sin(2\sin(\gamma/2))\,\tilde I(\gamma)},
\end{aligned}
\end{equation}
and let $\mathbf R(t)=\{R_i(t)\}$ be a global solution to \eqref{A-2}. Then there exist a unique equilibrium $\mathbf R^\infty\in (B_\gamma)^N$ such that
\[
\|\mathbf R(t)-\mathbf R^\infty\|_{\ell^1}\lesssim e^{-\lambda_1 t},
\qquad t\to\infty.
\]
\end{theorem}

\begin{proof}
By Corollary~\ref{C3.1}, there exists $t_0\ge 0$ such that
\[
\mathbf R(t)\in (B_\gamma)^N
\qquad \text{for all } t\ge t_0.
\]
We have proved the existence of an equilibrium $\mathbf R^\infty\in (B_\gamma)^N$. Since $\mathbf R^\infty$ is a stationary solution of \eqref{A-2}, Theorem~\ref{T3.1} applied to $\mathbf R(t)$ and $\mathbf R^\infty$ gives
\[
\|\mathbf R(t)-\mathbf R^\infty\|_{\ell^1}
\le
e^{-\lambda_1 (t-t_0)}
\|\mathbf R(t_0)-\mathbf R^\infty\|_{\ell^1},
\qquad t\ge t_0.
\]
This proves exponential convergence.

One can prove the uniqueness of equilibrium in $B_\gamma$ as follows:
Let $\widetilde{\mathbf R}^\infty\in(B_\gamma)^N$ be another equilibrium. Applying Theorem~\ref{T3.1} to the two stationary solutions $\mathbf R^\infty$ and $\widetilde{\mathbf R}^\infty$, we obtain
\[
\|\mathbf R^\infty-\widetilde{\mathbf R}^\infty\|_{\ell^1}
\le
e^{-\lambda_1t}\|\mathbf R^\infty-\widetilde{\mathbf R}^\infty\|_{\ell^1}
\qquad \forall\, t\ge 0.
\]
Hence $\mathbf R^\infty=\widetilde{\mathbf R}^\infty$.
\end{proof}

As an alternative relaxation result, if we use Corollary~\ref{C3.2} instead of Corollary~\ref{C3.1}, we obtain the following improvement.

\begin{corollary}\label{C4.1}
Suppose that the initial data and parameters satisfy $({\mathcal F}_A1)$ and, in addition, there exists an index set $K$ of cardinality $N_0$ such that $R_l^0 \in B_{\gamma_0}$ for every $l\in K$, and
\begin{equation}\notag\label{D-16-4}
\begin{aligned}
&
R_i^0 \in B_{\Gamma_i(\gamma)}, \quad \forall\, i\in[N], \quad \lambda_1:= \kappa(\cos \gamma\,\tilde I(\gamma)-\gamma\,\text{Lip*}I) > 0,\\
&\kappa > \kappa_c(\gamma,N_0)
:=
\frac{N}{N_0}\frac{\max_{j\in[N]}\|\Omega_j\|}{\sin(2\sin(\gamma/2))\,\tilde I(\gamma)},
\end{aligned}
\end{equation}
and let $\mathbf R(t)=\{R_i(t)\}$ be a global solution to \eqref{A-2}. Then there exist a unique equilibrium $\mathbf R^\infty\in(B_\gamma)^N$ such that
\[
\|\mathbf R(t)-\mathbf R^\infty\|_{\ell^1}\lesssim e^{-\lambda_1 t},
\qquad t\to\infty.
\]
\end{corollary}

\begin{remark}
For the high-dimensional Kuramoto model, complete state synchronization can occur even for a heterogeneous ensemble; see \cite{shi2023complete}. In contrast, for the high-dimensional Winfree model on $SO(n)$, complete state synchronization forces identical natural frequencies. Indeed, if
\[
R_i(t)\equiv R(t)
\qquad \text{for all } i\in[N],
\]
then subtracting the equations for indices $i$ and $j$ gives
\[
(\Omega_i-\Omega_j)R(t)=0.
\]
Since $R(t)\in SO(n)$ is invertible, it follows that $\Omega_i=\Omega_j$ for all $i,j\in[N]$. Thus complete state synchronization is possible only in the identical-oscillator case.
\end{remark}

\section{Homogeneous Winfree ensemble}\label{sec:5}
\setcounter{equation}{0}
In this section, we study the long-time behavior of the identical Winfree ensemble on $SO(n)$ whose dynamics is governed by the Cauchy problem for the following system:
\begin{equation}\label{E-1}
\begin{cases}
\displaystyle \dot R_i = \Omega R_i + \frac{\kappa}{2} (I_n-R_i^2) \Big( \frac{1}{N}  \sum_{j=1}^N I(R_j) \Big),\quad t >0,\\
\Omega \in \mathfrak{so}(n),\quad R_i^0 \in SO(n), \quad \forall~ i \in [N].
\end{cases}
\end{equation}
More precisely, we derive the following two assertions:
\vspace{0.1cm}
\begin{enumerate}
\item
System \eqref{E-1} exhibits complete state synchronization.
\vspace{0.1cm}
\item
The set of equilibria of \eqref{E-1} can be uncountable.
\end{enumerate}

\subsection{Complete state synchronization} \label{sec:5.1}
We first introduce the second sufficient framework $({\mathcal F}_B)$ in terms of initial data, and parameters:
\vspace{0.1cm}
 \begin{itemize}
 \item
 $({\mathcal F}_B1)$:~Parameters $\beta, \gamma_0$ and $\gamma$ satisfy 
 \[
0 <  \beta  < \frac{\pi}{2}, \quad  0 < \gamma_0 < 2\sin \Big (\frac{\beta}{2} \Big), \quad  \gamma = 2\arcsin \Big (\frac{\gamma_0}{2} \Big ) < \beta.
 \]
 Recall that $\beta$ is the upper bound of the support of $\tilde I$ (see \eqref{B-6}).
 \vspace{0.2cm}
 \item
 $({\mathcal F}_B2)$:~Coupling strength is sufficiently large such that 
 \[  \kappa  > \frac{ \|\Omega \|}{\tilde{I}(\gamma)\sin \Big(2\sin \frac{\gamma}{2} \Big )}. \]
 \item
  $({\mathcal F}_B3)$:~Initial data is confined in some Ball:
  \[
   R_i^0 \in B_{\gamma_0}, \quad \forall~i\in [N].
  \]
  \end{itemize}  
Note that $({\mathcal F}_B2) - ({\mathcal F}_B3)$ are basically the same with $({\mathcal F}_A2) - ({\mathcal F}_B3)$, and $({\mathcal F}_B1)$ is stronger than $({\mathcal F}_A1)$.
Therefore, $({\mathcal F}_B)$ implies $({\mathcal F}_A)$. 
\begin{theorem} \label{T5.1}
[Emergence of complete state synchronization]
Suppose that initial data and system parameters satisfy the framework $({\mathcal F}_B1) - ({\mathcal F}_B3)$, and let $\mathbf{R}=\{R_i\}_{i=1}^N$ be a global solution to \eqref{E-1}. Then, for a positive constant $\lambda_2 = \kappa \cos \gamma\, \tilde I(\gamma) > 0$ such that 
\begin{equation*}
    \|R_i(t) - R_j(t)\| \le e^{-\lambda_2 t}\|R_i^0 - R_j^0\|, \quad t >0.
\end{equation*}
\end{theorem}
\begin{proof} 
Recall that, we set 
\[
\langle I(t) \rangle := \frac{1}{N}\sum_{i=1}^N I(R_i(t)) = \frac1N \sum_{i=1}^{N} \tilde I(d(I_n, R_i(t))).
\]
Since we have $({\mathcal F}_A)$ and Proposition~\ref{P3.1}, 
\[ R_i(t)\in \overline{B_\gamma} \quad \mbox{for all $t\ge 0$}. \]
We use $\eqref{E-1}_1$ to get 
\[ \frac{d}{dt} (R_i - R_j) = \Omega(R_i-R_j)-\frac{\kappa}{2} \langle I \rangle (R_i^2-R_j^2).
\]
Now we follow a similar argument in the proof of Theorem~\ref{T3.1} to get
\begin{equation} \label{E-2}
\frac{d}{dt} \|R_i - R_j\|^2 \leq -2\cos \gamma\, \kappa \langle I \rangle \|R_i - R_j\|^2 = -2 \lambda_2  \|R_i - R_j\|^2.
\end{equation}
Since $\tilde I$ is decreasing and $d(I_n, R_k(t))\le \gamma<\beta,$ one has 
\begin{equation} \label{E-3}
 \langle I(t) \rangle = \frac{1}{N} \sum_{k=1}^{N} \tilde I(d(I_n, R_k(t))) \ge  \tilde I(\gamma)  >0.
\end{equation}
We combine \eqref{E-2} and \eqref{E-3} to find 
\begin{equation*}
\frac{d}{dt} \|R_i - R_j\|^2 \le -2\kappa  \tilde I(\gamma) \cos \gamma \|R_i - R_j\|^2.
\end{equation*}
Gronwall's lemma then gives the desired estimate:
\begin{equation*}
\|R_i(t) - R_j(t)\| \le \|R_i^0 - R_j^0\| \exp(-\lambda_2 t).
\end{equation*}
\end{proof}
\begin{remark}
One can actually show the exponential complete state synchronization from exponential $\ell^1$-stability in Theorem~\ref{T3.1}, for the indentical oscillator case. 
Indeed, by the symmetry, if $\mathbf{R}=\{R_1,\dots R_N\}$ is a global solution satisfying the conditions of Theorem~\ref{T3.1}, then the configuration $\mathbf{S}$ obtained by  switching $R_i$ and $R_j$ in $\mathbf{R}$ is also a solution satisfying the same conditions. By $\ell^1$-stability between $\mathbf{R}$ and $\mathbf{S}$, 
\[ \|R_i-R_j\|\le e^{-\lambda_1 t}\|R_i^0-R_j^0\|, \quad t > 0. \]
Thus, under the assumptions of Theorem~\ref{T3.1}, one obtains exponential complete state synchronization as well. Nevertheless, we state Theorem~\ref{T5.1} separately because, in the identical-frequency case, it gives a sharper decay rate. More precisely,
\[
\kappa\bigl(\cos\gamma\,\tilde I(\gamma)-\gamma\,\mathrm{Lip}^*I\bigr)=\lambda_1
<
\lambda_2=\kappa\cos\gamma\,\tilde I(\gamma).
\]
It also allows a weaker assumption, because the term involving $\mathrm{Lip}^*I$ disappears when the natural frequencies are identical.
\end{remark}

\vspace{0.5cm}

\subsection{Cardinality of equilibria} \label{sec:5.2}
Next, we classify a set of equilibrium for the Winfree model for $SO(n)$ for a homogeneous ensemble. The equilibrium $R_i^\infty$ must satisfy
\begin{equation} \label{E-4}
\Omega = \frac{\kappa \langle I^{\infty} \rangle}{2}(R_i^{\infty} - (R_i^{\infty})^T),
\quad\text{where}\quad\langle I^\infty\rangle := \frac1N\sum_{j=1}^N I(R_j^\infty).
\end{equation}
We consider two cases:  
\[ \quad \Omega=0 \quad \mbox{or} \quad  \Omega \ne 0. \]
Note that we do not restrict to the case for $R_i^{\infty} \in B_\gamma \subset B_{\frac\pi 2}$. \newline

\noindent $\bullet$~Case A: Suppose that 
\[  \Omega=0. \]
In this case, it follows from \eqref{E-4} that 
\[ \langle I^{\infty} \rangle (R_i^{\infty} - (R_i^{\infty})^T) = 0. \]
If \(\langle I^{\infty} \rangle = 0\), since \(I\ge0\), every $I(R_i^{\infty})$ must be zero:
\[ I(R_i^{\infty}) = 0 \quad \implies \quad  d(I_n, R_i^{\infty}) \ge \beta \quad \mbox{for all $i\in [N]$.}\]

\noindent If \(\langle I^{\infty} \rangle > 0\), then we have
\[ R_i^{\infty} = (R_i^{\infty})^T, \quad \mbox{i.e.,} \quad (R_i^{\infty})^2 = I. \]
This means that every eigenvalue of $R_i^{\infty}$ is either 1 or -1, and since the multiplicity of $-1$ must be even, the set of equilibria can be decomposed into two sets:
\[
\mathcal{E} = \mathcal{A} \cup \mathcal{B},
\] where
\begin{align*}
    \mathcal{A} &= \{(R_1, R_2, \dots, R_N) \in SO(n)^N : d(I_n, R^{\infty}_i) \ge \beta, \ \forall i \in [N]\}, \\
    \mathcal{B} &= \{(R_1, R_2, \dots, R_N) \in SO(n)^N : \text{every eigenvalue of } R^{\infty}_i \in \{1,-1\}, \\ 
    &\text{ and the multiplicity of } -\!1 \text{ is even,} \quad \forall\, i \in [N]\}.
\end{align*}

\vspace{0.5cm}

\noindent $\bullet$~Case B: Suppose that $\Omega\neq 0$.
Then
\[
\frac{R_i^{\infty} - (R_i^{\infty})^T}{2}=\frac{\Omega}{\kappa \langle I^{\infty} \rangle}.
\]
Then as we did in Section~\ref{sec:4}, we first write $\Omega$ as
\begin{equation}\notag\label{E-5}
    \Omega = P^T \text{diag}(\lambda_{1} J, \cdots, \lambda_{m}J, O_{n-2m}) P, \quad P\in O(n).
\end{equation}
for some postive constants $\lambda_1, \cdots, \lambda_m$. Then we may write $R_i\sim \Lambda_i$ where $\Lambda_i$ is defined in \eqref{D-16-2}, where $\theta_{ij}$ now satisfies
\[\{\sin^2\theta_{i1}, \cdots, \sin^2\theta_{im}\} = \Big\{\frac{\lambda_1^2}{\kappa^2 \langle I^{\infty} \rangle^2}, \cdots, \frac{\lambda_m^2}{\kappa^2 \langle I^{\infty} \rangle^2 } \Big\}\]
as a multisets. Note that since the condition $\theta_{ij}<\frac\pi 2$ are now removed, so we can not guarantee the relation like \eqref{D-17-1}.

\begin{remark}
According to the proof of existence and uniqueness of equilibrium without using $\ell^1$-stability in Section 4.3, for $\kappa > \frac{\| \Omega \|}{\sin \gamma\, \tilde{I}(\gamma)}$, we can prove the existence of a solution when 
\[ \theta^{\infty}_{ij} = \arcsin \frac{\lambda_j}{\kappa \langle I^{\infty}\rangle} \quad \mbox{for all $i, j \in [N]$}. \] 
We now discuss a different sufficient condition that also allows the other branch.
Let the choice of $\theta_{ij}$ be encoded by $\delta_{ij}\in\{0,1\}$:
\[
\theta^{\infty}_{ij}(x)=
\begin{cases}
\arcsin\!\big(\frac{|\lambda_j|}{\kappa x}\big), & \delta_{ij}=0,\\[4pt]
\pi-\arcsin\!\big(\frac{|\lambda_j|}{\kappa x}\big), & \delta_{ij}=1,
\end{cases}
\qquad x \ge x_*:= \frac{\lambda_*}{\kappa}=\frac{\max_{1\le j\le m}|\lambda_j|}{\kappa},
\]
and we set
\[
r_i(x):=\Big(\sum_{j=1}^m \theta^{\infty}_{ij}(x)^2\Big)^{1/2},
\qquad
T(x):=\frac1N\sum_{i=1}^N \tilde I\big(r_i(x)\big).
\]
Then the fixed-point equation is 
\[ x=T(x). \]
Since $x -T(x)\to \infty$ as $x \to \infty$, so a sufficient condition for the existence of a fixed point is
\[
T(x_*)>x_*.
\]

Uniqueness is, however, much more complicated.
This is due to the fact that $r_i$ has a singularity at $x_*$, and even when $\beta$ is sufficiently small so that every branch takes $\delta_{ij}=0$, the map $T$ is not contractive.

In fact, one can construct $\tilde I$ such that the solution of the fixed point equation is uncountably many, even for the principal branch case ($\delta_{ij}=0$ for all $i\in [N],\ j\in [m]$). \newline

In principle branch case, $r_i(x)$ is now independent of $i$, so denote it by $r(x)$.
\[r_i= \sqrt{\sum_{j=1}^m (\theta_{ij}^\infty(x))^2} = \sqrt{\sum_{j=1}^m \arcsin^2\Big(\frac{\lambda_j}{\kappa x}\Big)}\eqqcolon r(x).\] The fixed point equation is now $x = \tilde I(r(x))$.
Note that $r$ is defined on $(\lambda_*/\kappa, \infty)$, and decreasing on there.
Choose $x_0\in (\lambda_*/\kappa, 1)$ and $\beta\in (r(x_0), \pi)$.
Define
\[
\tilde I(x)=\begin{cases}
1 & x\in [0,r(1)],\\
r^{-1}(x)& x\in [r(1), r(x_0)],\\
\text{interpolation of end points} & x\in [r(x_0), \beta],\\
0 & x\in [\beta, \infty).
\end{cases}
\]
Then one can check that $\tilde I$ is decreasing, Lipschitz, and has support $[0,\beta]$ so satisfies the assumptions in Section \ref{sec:4}.
Moreover, by construction
\[
\tilde I(r(x))= x \quad\text{for all }x\in[x_0, 1],
\]
which yields a continuum of fixed points.
\end{remark}

\section{Conclusion}\label{sec:6}
\setcounter{equation}{0}
In this work, we have provided a Winfree  model on $SO(n)$ and studied its emergent dynamics. 
We generalized the influence function to make it depend on the geodesic distance on $SO(n)$. We established a positively invariant trapping region, proved a leader--follower mechanism under sufficiently strong coupling, derived exponential $\ell^1$-stability, and obtained existence, uniqueness, and exponential convergence toward an equilibrium in a suitable neighborhood of the identity.
In the identical-oscillator regime, we proved exponential complete state synchronization and described the corresponding equilibrium configurations.

There are several interesting problems that have not been addressed in this paper. For example, it would be interesting to study constants of motion of the homogeneous Winfree matrix model on $SO(n)$, or we can think of other choice of the influence function. 
We leave these questions for a future work.

\appendix

\section{Proof of Lemma \ref{L2.3} and \ref{L4.1}} \label{App-A}
\setcounter{equation}{0}

\subsection{Proof of Lemma \ref{L2.3}}\label{Ap-A-1}
Below, we provide detailed proofs of each assertion in Lemma \ref{L2.3} one by one. \newline

\noindent (1) For $R \in SO(n)$, let $(\lambda, v)$ be the eigenvalue-eigenvector pair of $R$. Then the orthogonality of $R$ implies that $|\lambda|=1$. Hence there exists $\theta \in [-\pi, \pi]$ such that
\[\lambda = e^{\mathrm{i}\theta}.\]
If $\theta \ne 0$ and $\theta \ne \pi$, then since $R$ is a real matrix, $(\bar \lambda, \bar v)$ is also another eigenvalue-eigenvector pair of $R$. From the fact that $\bar\lambda = e^{-\mathrm{i}\theta}$, we obtain two eigenpairs $\{(e^{\mathrm{i} \theta}, v), (e^{-\mathrm{i} \theta}, \bar v)\}$. \newline

\noindent If $\theta = 0$, this corresponds to the eigenpair $(1, v)$. \newline

\noindent If $\theta = \pi$, this corresponds to the eigenpair $(-1, v)$. Since $\mathrm{det}(R) = 1$, the multiplicity of $-1$ must be even. 
Finally, we may put $-\theta$ instead of $\theta$ when $\theta<0$, so we may assume $\theta\in(0, \pi].$ 
This implies the desired first assertion. 

\vspace{0.5cm}

\noindent(2) Since $R$ is orthogonal, it is normal and therefore it is unitarily diagonalizable. Let
\[ R = V^\dagger\mathrm{diag}(e^{\mathrm{i}\theta_1}, e^{-\mathrm{i}\theta_1}, \cdots, e^{\mathrm{i}\theta_m}, 1, \cdots, 1)V, \]
where $V^\dagger$ is the Hermitian conjugate of $V$. For any non-real eigenvalue $\lambda = e^{\mathrm{i}\theta}$ and its corresponding eigenvector $v$, let $v = p + \mathrm{i}q$ for two real vectors $p$ and $q$. Then it satisfies
\begin{equation} \label{Ap-1}
Rp   = (\cos\theta )p - (\sin\theta) q, \quad Rq =  (\sin\theta) p + (\cos\theta) q.
\end{equation}
Now, we claim that $p$ and $q$ are linearly independent. Suppose the contrary holds, i.e., suppose that $p$ and $q$ are linearly dependent: for some $k\in \bbr$, $p=kq$.
Then \eqref{Ap-1} gives
\[ k\Big((\sin\theta)kq + (\cos\theta)q\Big) =  (\cos\theta)kq -(\sin\theta)q,\]
hence
\[(k^2+1) q\sin\theta =0 \quad \Longrightarrow \quad q\sin \theta=0.\]
However, if $\sin\theta=0$ or $q=0$, $\lambda$ becomes real which is a contradiction.
Hence $p, q$ are linearly independent. This shows that the vector space $W\coloneqq \mathrm{span}\{p, q\}$ is an invariant set of $R$, and $R\Big|_W$ is a rotation map with angle $\theta$ on $R$.
Thus, we may pick orthonormal basis $\beta=\{e_1, e_2\}$ so that \[[R]_{\beta}^{\beta} = \begin{pmatrix}
 \cos\theta & -\sin\theta	\\
 \sin\theta & \cos\theta
 \end{pmatrix}.\]
Then we repeat this process to all non-real eigenvalues to obtain the desired block diagonal result.

\vspace{0.5cm}

\noindent (3) This result is the immediate consequence of
\[\exp\begin{pmatrix}
0 & -\theta \\ \theta & 0
\end{pmatrix} = 
\begin{pmatrix}
	\cos\theta & -\sin\theta \\
	\sin\theta & \cos\theta
\end{pmatrix}.
\]

\vspace{0.5cm}

\noindent (4) Let $\Omega = Q^T H Q$ be the real Schur decomposition of $\Omega$. i.e. $Q$ is orthogonal matrix, and $H$ is upper quasi diagonal:
\[H = \begin{pmatrix}
    B_1 &*  & \cdots & * \\
    0& B_2 &\cdots &* \\
    \vdots & \vdots &\ddots& * \\
    0 & 0 & 0 & B_k
\end{pmatrix},\]
where $B_j$ are $1\times 1$ or $2\times 2$ block matrix with form
\[\begin{pmatrix}
    a & -b \\
    b & a
\end{pmatrix}.\]
Then the antisymmetric property of $\Omega$ eliminates all $1\times 1$ block, and forces $2\times 2$ block matrix to have the form 
\[\begin{pmatrix}
    0 & -\lambda \\
    \lambda & 0
\end{pmatrix} = \lambda J.\]
For some scalar $\lambda$. Since $\lambda J$ and $-\lambda J$ are orthogonally similar, we may assume the positivity of $\lambda$. This concludes the proof.
\vspace{0.5cm}

\noindent (5)
For every matrix Lie group $G$, it is well known that matrix exponential coincides with the exponential map of $G$.
Then by Hopf-Rinow theorem, there exists a globally length-minimizing geodesic $\gamma(t)$ connecting $I_n$ and $A$, i.e. , the length of $\gamma(t)$ is equal to the infimum of the lengths of all piecewise differential curve joining $I_n$ and $A$. 

It is also well known that
\[e^B = \begin{pmatrix}\cos\theta & -\sin\theta \\ \sin \theta & \cos\theta \end{pmatrix} \quad \Longleftrightarrow \quad B = \begin{pmatrix}
0 & -(\theta + 2\pi m)\\
\theta + 2\pi m & 0
\end{pmatrix}, \quad m\in \bbz.\]
This shows that, for the matrix $X$ defined in Lemma~\ref{L2.3} (3), the geodesic $\gamma(t)$ that connects $I_n$ and $A$ has the form
\[\gamma(t) = P^T \exp(t \tilde X)P,\]
where 
\[\tilde X = X + 2\pi \text{diag}(n_1 J, \cdots n_m J, O_{n-2m}), \quad n_j \in \bbz.\]
Note that the length of $\gamma(t)$ is equal to $\|\tilde X\|$, and we have
\[\|\tilde X\| = \sqrt{\frac{1}{2}\tr (\tilde X^T \tilde X)} = \sqrt{\sum_{j=1}^{m}(\theta_j+2\pi n_j)^2},\]
which is clearly minimized when $n_1 = \cdots =n_m = 0.$ This implies the desired results
\[\gamma(t) = P^T \exp (X) P, \quad d(I_n, A) = \sqrt{\sum_{j=1}^{m}\theta_j^2}.\]
For the last statment, let $C(A, B)$ be every curve connecting $A$ and $B$. i.e.
\[C(A, B) = \{\gamma:[0,1]\rightarrow SO(n)\, \big|\, \gamma(0)=A,\ \gamma(1)=B \, \}.\]
Then for $R, S\in SO(n)$, we construct a map $\varphi:C(R, S)\rightarrow C(I_n, R^T S)$ as follows:
\[\varphi(\gamma)(t) = R^T \gamma(t).\]
Then since the Riemannian metric $g$ is bi-invariant, the length of $\varphi(\gamma)$ and $\gamma$ is the same. Hence
\[d(R, S) = d(I_n, R^T S).\]

\noindent (6)
It is enough to show when $C=0$.
Since $A$ is positive semidefinite, we can compute a positive semidefinite square root $A^{1/2}$.
Then by symmetry
\[
\tr(AB)=\tr((A^T)^{1/2}BA^{1/2})\ge 0,
\]
since $(A^T)^{1/2}BA^{1/2}$ is positive semidefinite.This completes the proof of Lemma \ref{L2.3}.

\subsection{Proof of Lemma \ref{L4.1}} \label{App-A-2}
Let
\[
A=\mathrm{diag}(a_1J,\cdots,a_pJ,O_{n-2p}), \qquad
B=\mathrm{diag}(b_1J,\cdots,b_qJ,O_{n-2q}),
\]
and choose $P\in O(n)$ such that
\[
A=P^TBP.
\]
Then
\[
A^TA=(P^TBP)^T(P^TBP)=P^TB^TPP^TBP=P^TB^TBP,
\]
which shows that \(A^TA\) and \(B^TB\) are orthogonally similar. Hence they have the same eigenvalues, counted with multiplicities. On the other hand,
\[
A^TA=\mathrm{diag}(a_1^2I_2,\cdots,a_p^2I_2,O_{n-2p}),
\qquad
B^TB=\mathrm{diag}(b_1^2I_2,\cdots,b_q^2I_2,O_{n-2q}).
\]
Therefore, we have
\[p=q\quad \text{and}\quad \{a_1^2, \cdots, a_p^2\} = \{b_1^2, \cdots, b_q^2\}\]
as multisets. This completes the proof.

\end{document}